\documentclass[12pt]{article}
\usepackage{amsmath,amsfonts,amssymb,amsthm}
\usepackage{graphicx}
\usepackage{hyperref}
\usepackage{geometry}
\usepackage{tikz}
\usetikzlibrary{decorations.pathmorphing}
\usetikzlibrary{calc}
\usetikzlibrary{arrows.meta}
\newtheorem{theorem}{Theorem}[section]
\newtheorem{proposition}[theorem]{Proposition}
\newtheorem{lemma}[theorem]{Lemma}
\newtheorem{corollary}[theorem]{Corollary}
\newtheorem{definition}[theorem]{Definition}
\newtheorem{remark}[theorem]{Remark}
\newtheorem{example}[theorem]{Example}
\newcommand{\diff}{\text{Diff}}
\newcommand{\homeo}{\text{Homeo}}
\newcommand{\flux}{\mathcal{F}}
\newcommand{\ahom}{\mathfrak{AHom}}
\newcommand{\rdelta}{\mathbf{R}_\delta}

\newcommand{\iso}{\text{Iso}}
\newcommand{\homo}{\text{Hom}}

\newcommand{\zcal}{\mathcal{Z}}
\newcommand{\ocal}{\mathcal{O}}
\newcommand{\pcal}{\mathcal{P}}

\begin{document}
		\begin{center}
		{\bf\large Extending the flux homomorphism to volume-preserving homeomorphisms}
		\\[0.5cm]
		{ St\'{e}phane Tchuiaga $^a$ \footnote{$^*$Corresponding Author} \\[2mm]
			$^a$Department of Mathematics of the University of Buea, 
			South West Region, Cameroon\\[2mm]
			{\tt E-mail: tchuiagas@gmail.com}} \\[2mm]
	\end{center}%
	\vspace*{0.5cm}
	\begin{abstract}
		This paper extends the flux homomorphism to volume-preserving homeomorphisms. 
		A surprising  $(C^0, \delta)-$rigidity result where the extended flux groups coincide with the standard flux group is proved. The introduced tools, which also include a Poincaré duality with Fathi’s mass flow and a norm on the group of volume-preserving homeomorphisms, indicate a potential for  new flexibility in the behavior of homeomorphisms. This flexibility could have implications for rigidity results in symplectic/cosymplectic geometry, particularly those concerning Lefschetz manifolds:  Any finite energy symplectic homeomorphism of $(T^2, \omega)$ with trivial flux, is a finite energy Hamiltonian homeomorphism of $(T^2, \omega)$. We discuss the cohomology groups $H^\ast(Homeo_0(M,\Omega), \mathcal{C}(M, \mathbb{R}) )$ of $ Homeo_0(M,\Omega)$ with coefficients in $  \mathcal{C}(M, \mathbb{R})$. 
	\end{abstract}
	

\section{Introduction}
The study of dynamical systems and {\it control} often involves understanding the long-term behavior of systems, especially in the presence of constraints, conservation laws, and perturbations. {\it In control theory, guaranteeing stability and achieving desired performance under various uncertainties is a fundamental objective.} A central tool for analyzing such systems is the concept of a flow, a continuous transformation that describes the system's evolution over time. Within this domain, transformations that preserve geometric structures like volume or symplectic forms play a key role, as they reflect  conservation laws in many other systems. The flux homomorphism is a well-established invariant used to study such transformations, particularly within the realm of diffeomorphisms. It measures the "quantity" of the flow generated by an isotopy, and it plays a crucial role in understanding the structure of diffeomorphism groups.\\ 

However, in many practical applications, we encounter systems where the regularity is limited, or where the transformations are continuous (homeomorphisms) but not necessarily differentiable.\\ {\it This is common in systems with friction, hysteresis, or switching, where the control inputs may not be smooth.}\\ This paper addresses a key challenge: extending the well-established notion of the flux homomorphism, traditionally defined for smooth diffeomorphisms, to the class of volume-preserving homeomorphisms. This has important ramifications because there are many systems that do not satisfy differentiability conditions, but the topological properties of the flow are well-defined.

\begin{theorem}\label{Theo-4} \cite{Mull-1, J-C}
	Let $M$ be a closed connected smooth manifold of dimension $n$. If $n\leq 3$, then any homeomorphism can be uniformly approximated by diffeomorphisms. If $n\geq 5$, then a homeomorphism $h$ of $M$ can be uniformly approximated by a diffeomorphism $\phi$ if and only if $h$ is isotopic to a diffeomorphism.
\end{theorem}

From a control perspective, this theorem implies that many complex systems can be approximated by simpler, smooth models, which is a cornerstone for performing predictions. However, some caution must be observed, and in the present paper we focus on aspects that are valid when we remove all smoothness on the system.  The above theorem holds if we replace homeomorphism (resp. diffeomorphism) with volume-preserving homeomorphism (resp. volume-preserving diffeomorphism) \cite{J-C}, \cite{Mull-1}. The isotopies concerned in Theorem \ref{Theo-4} are continuous maps 
$\Phi : [0, 1] \rightarrow Homeo(M), t\mapsto \phi_t$. If the homeomorphism $h$ is volume-preserving, then it can be uniformly approximated by a volume-preserving diffeomorphism \cite{J-C}, \cite{Mull-1}. So, it is known that on a closed connected oriented manifold, any continuous path $\Psi := \{\psi^t\}$ in $Homeo(M)$ with $\psi^0 = id_M$ can be uniformly approximated by a continuous
isotopy $\Theta : = \{\theta_t\}$ in $Diff^\infty(M)$. In terms of invariant, this raises the following  question.

\subsection*{Question:} Is there a general invariant for volume-preserving homeomorphisms in high dimension which coincides with the usual flux homomorphism when restricted to the group of volume-preserving diffeomorphisms in an optimal way (by optimal, mean respecting a general setting of the Fathi-Poincar\'e duality result)?\\

One goal of the present paper is to answer the above question. The contribution of this work is two-fold. Firstly, we introduce a method for extending the flux homomorphism to volume-preserving homeomorphisms via obstructions (see equation (\ref{Functor3})). Secondly, we prove a rigidity result (Theorem \ref{Theo}) demonstrating that all the topological flux groups associated with the extended flux coincides with the standard flux group. We also introduce a norm on the group of volume-preserving homeomorphisms (see equation \ref{Fixpoints-metric1}) which measures how much a transformation distorts differential forms, which are closely related to concepts of flow, circulation, and geometry.\\
 From a control perspective, the smaller the norm, the smaller a perturbation to the trajectory of a system will be after the system returns to its initial position. This makes norms of this type relevant in areas ranging from fluid dynamics (turbulent flows, mixing, or vorticity), in mechanics (systems with conserved quantities, or the chaoticity of dynamical systems), and theoretical physics.  In control theory, systems are often governed by conserved quantities, like the volume of the phase space, and understanding how to change the dynamics without changing the conserved quantities will be very useful. The tools developed also include a Poincaré duality with Fathi’s mass flow. These tools indicate a potential for  new flexibility in the behavior of homeomorphisms, \textbf{which can provide more degrees of freedom for controllers to be designed,} which could have  implications for understanding both the stability and control of dynamical systems.\\

This paper is organized as follows: In Section 2, we construct the extended flux invariant using gap functions. In section 3, we study volume preserving homeomorphisms, and derive the core results. Section 4 discusses a metric on the space of volume preserving homeomorphisms and section 5 concludes the paper.

	\section{Prelimimaries}
	In this section, we introduce the key concept of "gap functions," denoted by $\mathbf{R}_\delta(H, \alpha)$. These functions are fundamental to our approach, as they allow us to define an extended flux homomorphism for volume-preserving homeomorphisms, which are transformations that are continuous but not necessarily differentiable. The challenge is that the standard flux homomorphism relies on the differentiability of the map, but we want a definition that also applies in the $C^0$ context. To achieve this, we utilize the fact that homeomorphisms can be approximated by diffeomorphisms (smooth transformations), which is known if the dimension is less than or equal to 3, or when the homeomorphisms are isotopic to diffeomorphisms (Theorem 1.1).  The gap functions quantify the differences arising from these approximations, and, as we shall see, they will allow us to extend the notion of a flux to the $C^0$ context.
\subsection{The gap functions:  $ \mathbf{R}_\delta(H, \alpha)$}
Let $r(g)$ denote the injectivity radius of the Riemannian metric $g$ on $M$. For each $\delta \in ]0,r(g)[$, we may choose a path $\Theta$ such that $d_{C^0}(\Theta, \Psi) < \delta$. Here, $\Psi$ represents the original path of homeomorphisms that we want to analyze, and $\Theta$ represents a smooth approximation of this path. In other words, we can approximate a path of homeomorphisms by a path of diffeomorphisms, up to an arbitrarily small $C^0$ error. If $\psi^1$ is a volume-preserving diffeomorphism, then we can assume that $\theta_1 = \psi_1$.  On the other hand, the continuous path $\Theta : = \{\theta_t\}$ is homotopic relatively to fixed endpoints to a smooth path $\Phi = \{\phi^t\}$ in $\diff^\infty_0(M)$ \cite{TKS25}. This homotopy, $\Sigma^s_t$, connects the approximation $\Theta$ with a smooth isotopy $\Phi$ such that $\Sigma^1_t = \phi^t$, and $\Sigma^0_t = \theta^t$.  The condition of being "homotopic relatively to fixed endpoints" is essential because it ensures that the approximation respects the initial and final configurations of the transformation. This will be a key idea in the definition of the extended flux. Let $\mathfrak{AHom}_\delta(\Psi)$ denote the set of all smooth isotopies homotopic relatively to fixed endpoint to an approximate $\Theta$ of $\Psi$ such that $d_{C^0}(\Theta, \Psi)< \delta$. We shall denote the set $\mathfrak{AHom}_\delta(\Psi)$ by $\mathfrak{AHom}_\delta^\Omega(\Psi)$ in the case of volume-preserving \cite{Ban97}.
Fix $\delta$ in  $]0,r(g)[$, and for each closed $1-$form $\alpha$ on $M$, we define the gap function as follows:     	   

\begin{equation}\label{Functor3}
	\rdelta(H, \alpha)(p)  : =
	\left\{
	\begin{array}{lccc}
		( \flux_\alpha(H)(1))(p) &  \text{ if }  &  H\in \iso(M)\\
		( \flux_\alpha(\Phi)(1))(p) +  \int_{\chi_\delta^p}\alpha, & \text{ if } & H \in (\pcal(\homeo(M),\text{Id})\setminus \iso(M)),
	\end{array}
	\right.  
\end{equation}
where  $ \flux_\alpha
(\Phi)(t)  : = \int_0^{t}\left( (\Phi_s)^\ast(\iota_{\dot\Phi_s}\alpha)\right) ds$, $\Phi$ is any element in $\ahom_\delta(H)$, and $\chi_\delta^p$ is a minimizing geodesic from $H(1,p)$ to $\Phi(1,p)$. This geodesic will be needed to relate the orbits of $H$ and $\Phi$. The map $ \rdelta$ is well defined and does not 
depend on the choice of $\Phi$  in $\ahom_\delta(H)$, and the Hopf-Rinow theorem guarantees the existence and uniqueness of the minimizing geodesic from  $\Phi(1,p)$ to $H(1,p)$. For each $p\in M$, the orbit $\ocal_p^H$ can be written as a collection $\bigsqcup_{i\in I}C_i(p)$ of smooth curves $C_i(p)$, and for all closed $1-$form $\alpha$, the integral  $\int_{\ocal_p^H}\alpha$ is interpreted as $ \sum_{i\in I}\int_{C_i(p)}\alpha$. It follows that for each  $\Phi\in \ahom_\delta(H)$, we have 
\begin{equation}\label{Functor33}
	\int_{\ocal_p^H}\alpha = \int_{\ocal_p^\Phi}\alpha + \int_{\chi_\delta^p}\alpha =  ( \flux_\alpha(\Phi)(1))(p) +  \int_{\chi_\delta^p}\alpha.
\end{equation}  
Also, for each smooth curve $\gamma$, since the curves $s\mapsto H(1,\gamma (s) )$, $s\mapsto \Phi(1,\gamma (s) )$, and geodesics $\chi_\delta^{\gamma(0)}$, $\chi_\delta^{\gamma(1)}$ form the boundary of a $2-$chain, we equally have  
\begin{equation}\label{Functor44}
	\int_{ H(1,.)\circ \gamma}\alpha  - 	\int_{\Phi(1,.)\circ \gamma}\alpha= 
	\int_{\chi_\delta^{\gamma(0)}}\alpha- \int_{\chi_\delta^{\gamma(1)}}\alpha.
\end{equation}
\begin{figure}[ht] 
	\centering
	\begin{tikzpicture}[scale=1.0]
		\draw[thick, blue] (0,0) to[out=30, in=150] (5,0);
		\node[above] at (0,0) {Id};
		\node[above] at (5,0) {$h$};
		\node[above, blue] at (2.5, 0.5) {$\Phi_t$};
		
		\draw[thick, red, decorate, decoration={snake,amplitude=0.1cm,segment length=0.2cm}] (0,-1) to[out=30, in=150] (5,-1);
		\node[below] at (0,-1) {Id};
		\node[below] at (5,-1) {$h$};
		\node[below, red] at (3,-0.5) {$H_t$};
		\foreach \x in {0.5,1.5,...,4.5}
		{
			\draw[dashed, gray] (\x,0) to[out=-60, in=60] (\x,-1);
		}
		\draw[dashed, gray] (0,0) to[out=-60, in=60] (0,-1);
		\draw[dashed, gray] (5,0) to[out=-60, in=60] (5,-1);
		
		\draw[->, dashed, gray, thick] (2,0) to[out=-60, in=60] (2,-1);
		\node [right, align=left] at (2,-0.5) {$\chi_\delta^p$};
		
		\begin{scope}[shift={(2,-1.7)},scale=0.5]
			\draw[red, thick] (0,0) -- ++(0.1, 0) -- ++(0, 0.4) -- ++(-0.1,0);
			\draw[blue, thick] (0,0) -- ++(0.1, 0) -- ++(0, 0.25) -- ++(-0.1,0);
			\draw[<->] (0,0.25) to (0,0.4);
			\node [right, align=left] at (0.1,0.3) {$\int_{\chi_\delta^p}\alpha$};
			
		\end{scope}
		
	\end{tikzpicture}
	\caption{This image shows a smooth path $\Phi_t$ (blue) and a non-smooth path $H_t$ (red) connecting two points.
		Several geodesics $\chi_\delta^p$ connect points on both paths, illustrating how the gap is measured.
		A close up shows the correction term.}
	\label{fig:gap_function} 
\end{figure}
	\subsection{Properties of the functions  $ \mathbf{R}_\delta(H, \alpha)$}
	
	\begin{proposition}\label{pro2}  Let $\alpha$ be a  closed $1-$form, $H \in \pcal\homeo_0(M)$ and  $\delta \in ]0,r(g)[$. Then 
		for all $x, y\in M$, and
		for each $\gamma\in \mathcal C^1(x\longrightarrow y)$, 
		the quantity
		$I_\alpha(h,x, y)$ coincides 
		with  $\rdelta(H,\alpha)(y) - 
		\rdelta(H,\alpha)(x)$.
	\end{proposition}
	\begin{proof} Pick any $\Phi :=(\phi_t)$ in  $\ahom_\delta(H)$, and $\gamma\in \mathcal C^1(x\longrightarrow y)$. Consider the smooth $2-$chain $ \bigcirc(\Phi, \gamma) := \{\phi_t(\gamma(s)): 0\leqslant s, t\leqslant 1\}$.  Applying Stokes theorem  gives 
		$\int_{\phi_1(\gamma)}\alpha - \int_{\gamma}\alpha = \int_{\ocal_{x}^\Phi}\alpha -\int_{\ocal_{y}^\Phi}\alpha.$
		The above equality together with (\ref{Functor44}) yields the desired result. 
	\end{proof}
	
	\begin{corollary}\label{cor00}   Let $\alpha$ be a  closed $1-$form, $H \in \pcal\homeo_0(M)$ with time-one map $h$ and  $\delta \in ]0,r(g)[$. Then 
		$ \int_{h\circ\gamma} \alpha   - \int_\gamma\alpha = \rdelta(H,\alpha)(\gamma(1)) - 
		\rdelta(H,\alpha)(\gamma(0)),$
		for all curve $\gamma$ in $M$. 
	\end{corollary}
	
	\begin{corollary}\label{cor0}   Let $\alpha$ be a  closed $1-$form, and $H , N\in \pcal \homeo_0(M)$ . The following hold. 
		\begin{itemize}
			\item $ \rdelta(H\circ N,\alpha) = \rdelta( N,\alpha) + 
			\rdelta(H,\alpha)\circ N$.
			\item $\rdelta(H^{-1},\alpha) = - \rdelta(H,\alpha)\circ H^{-1}$.
		\end{itemize} 
	\end{corollary}
	\begin{proof}
		This follows from  formula (\ref{Functor33}).
	\end{proof}
	\begin{lemma}\label{lem3} Fix $\delta \in ]0,r(g)[$.
		Assume that $H, H'\in \pcal\homeo_0(M)$ are 
		homotopic relatively to fix endpoints, and let $z\in M$ such that 
		$\ocal_z^H$ and 
		$\ocal_z^{H'}$ is piece-wise smooth. Then,  
		$\rdelta(H,\alpha) =  \rdelta(H',\alpha),$
		for each closed  $1-$form $\alpha$. 
	\end{lemma}
	\section*{Interpretation of \( I_\alpha(h, x, y) \)}  
	Let $H = (h_t)$ be a continuous path in $\text{Homeo}(M)$ such that $h_1$ is its time-one map, pick $x, y \in M$, 
	 consider curves $\gamma \in C^1(x \to y)$, and study for each closed $1-$forms  $\alpha$ consider the function
	\[
	I_\alpha(h, x, y) := \int_{h_1 \circ \gamma} \alpha - \int_{\gamma} \alpha.
	\]
	The integral \( \int_\gamma \alpha \) measures the ``work'' done by the 1-form \( \alpha \) along the curve \( \gamma \), while the integral \( \int_{h_1 \circ \gamma} \alpha \) measures the ``work'' done by \( \alpha \) along the deformed curve \( h_1 \circ \gamma \).  The difference \( I_\alpha(h, x, y) \) quantifies how much the deformation \( h_1 \) changes the action of \( \alpha \) along \( \gamma \). \\
	In Fathi's construction for the $2-$disk \cite{Fathi}, the angular variation measures how much the direction between pairs of points changes under a deformation.  
	Here, \( I_\alpha(h, x, y) \) generalizes this idea by replacing angles with the action of curves on closed 1-forms, which is a more natural invariant for manifolds. \\   

	 The functional \( I_\alpha(h, x, y) \) provides a way to quantify the global effect of a deformation \( h_1 \) on the manifold \( M \). It captures how much the deformation ``twists'' or ``shears'' the manifold, as measured by the action of curves on closed $1-$forms.   In the study of dynamical systems, \( I_\alpha(h, x, y) \) can be used to analyze the behavior of flows or diffeomorphisms on \( M \). For example, it can detect the presence of fixed points, periodic orbits, or other dynamical features.  Thus,  
	 \( I_\alpha(h, x, y) \) generalizes Fathi's angular variation invariant to the context of manifolds with closed $1-$forms in measuring the gap between the actions of curves \( \gamma \) and \( h_1 \circ \gamma \) on closed $1-$forms. This provides a  tool for quantifying the global effects of deformations.

Normalizing the volume form $\Omega$ such that $\int_M \Omega = 1$ and fixing a closed 1-form $\alpha$, then to each $h$ in the identity
component of $\text{Homeo}(M, \Omega)$, one can assign a continuous real valued function $\chi(h, \alpha)$ defined
on $M$ by
\[
\chi(h, \alpha)(z) := \frac{1}{\|\alpha \|_{L^2}}\int_M I_\alpha(h, z, \cdot) \Omega, \quad \text{for each } z \in M. \tag{1}
\]
where $I_\alpha(h, z, \cdot) : y \mapsto I_\alpha(h, z, y) \in \mathbb{R}$.
Then, for each non-trivial closed $1-$form $\alpha$,
we have the following well-defined mapping :
$\chi(\cdot,\alpha) : \text{Homeo}_0(M)\longrightarrow C^0(M,\mathbb{R}),$ such that
for each $h\in \text{Homeo}_0(M)$, we have
$$
\chi(h,\alpha)(z) : = \frac{1}{\|\alpha \|_{L^2}}\int_M I_\alpha(h,z, \cdot)\, \Omega.
$$
	\begin{corollary}\label{lem1}  Let $\alpha$ be a  closed $1-$form, $H \in \pcal\homeo_0(M)$ and  $\delta \in ]0,r(g)[$. If $h$ is the time-one map of $H$, then  for all $z\in M$, we have 
		$$ \|\alpha \|_{L^2}\displaystyle\chi(h,\alpha)(z) 
		= -\left(\rdelta (H,\alpha)(z) - \int_M \rdelta(H,\alpha)\, \Omega\right),$$ and the latter formula does not depend on the choice of the path $H$ in $\pcal\homeo_0(M)$ from the identity to $h$. 
	\end{corollary}
	\begin{proof} This follows from   Corollary \ref{cor00}. 
	\end{proof}
	\begin{corollary}\label{cor1}  Let $\alpha$ be a  closed $1-$form, $H \in \pcal\homeo_0(M)$ and  $\delta \in ]0,r(g)[$. Then the  integral
		$\displaystyle \int_M \rdelta(H,\alpha)\Omega$ is independent of the choice of any representative $\beta$
		in the de Rham cohomology class  $ [\alpha]$.
	\end{corollary}
	\begin{proof} Let $\beta\in [\alpha],$ i.e.,  
		$\alpha - \beta = df$ for some smooth function $f : M\longrightarrow\mathbb{R}$. Pick 
		$\Phi = (\phi_t)\in \ahom_\delta(H)$. Derive from \cite{HLGF} that  
		$\int_M\flux_\alpha(\Phi)(1)\Omega = \int_M\flux_\beta(\Phi)(1)\Omega$.  Similar argument gives $\int_M (\int_{\chi_\delta}\beta)\Omega = \int_M(\int_{\chi_\delta}\alpha )\Omega $. 
	\end{proof}
	\begin{corollary}\label{cor2} Let $\alpha$ be a non-trivial closed $1-$form. If $h_1,h_2\in \homeo_0(M)$, then \\
		$ \chi(h_1\circ h_2,\alpha) = \chi(h_2,\alpha) + \chi(h_1,\alpha)\circ h_2.$
	\end{corollary}
	
Note that $ \chi(h,\alpha)(z)$ is  capturing the cohomological effect of $h$ 
with respect to the closed $1-$form $\alpha$ 
at the point $z$.
\section{Volume-preserving homeomorphisms}
In this section, we delve into the properties of volume-preserving homeomorphisms and how the extended flux invariant defined in Section 2, denoted by $\widetilde{L}_\Omega^\delta$, relates to the usual flux homomorphism, when restricted to diffeomorphisms. This section will be dedicated to proving our main result. Recall that volume-preserving transformations are essential in the study of dynamical systems because they represent systems where the phase space volume is constant, a property that characterizes, for instance, Hamiltonian systems. Thus, the extended flux becomes a powerful tool for their analysis. Corollary \ref{cor0} and Corollary \ref{cor1} suggest that: Fix $\delta \in ]0,r(g)[$, 
each isotopy   
$H\in \pcal\homeo_0(M,\Omega)$ induces an element $
T_H^\delta\in \homo(H^{1}(M,\mathbb{R}),\mathbb{R}),$ defined as 
$T_H^\delta([\alpha]) : =  \displaystyle \int_M \rdelta(H,\alpha)\Omega$. 
By the Poincar\'e duality theorem, we have a group homomorphism \\
$\widetilde{L}_\Omega^\delta : \pcal\homeo_0(M, \Omega) \longrightarrow H^{(n-1)}(M,\mathbb{R})$ such that 
$ T_H^\delta([\alpha]) = \langle  [\alpha], \widetilde{L}_\Omega^\delta(H)\rangle $, where\\ 
$ \langle\cdot ,\cdot\rangle :  H^1(M,\mathbb{R})\times H^{(n-1)}(M,\mathbb{R}) \longrightarrow   \mathbb{R}, ([\alpha], [\beta]) \longmapsto  \displaystyle \int_{M}\alpha\wedge\beta,$
is the usual Poincar\'e pairing.
This map, $\widetilde{L}_\Omega^\delta$ serves as a natural extension of the usual flux homomorphism defined for diffeomorphisms. It extends the flux into the realm of continuous but non-differentiable volume-preserving transformations, and as we will show, the map is a non-trivial invariant.

This motivated the following factorization result that generalizes Proposition 2.4 of \cite{HLGF}. This result will be key to the proof of our main theorem.

\begin{proposition}\label{pro3}  Let $(M,\Omega)$ be a closed oriented manifold, $\alpha $ be a closed $1-$form, and  
	$ H\in \pcal\homeo_0(M)$. Fix $\delta \in ]0,r(g)[$.  Then, we have 
	$ \displaystyle T_H^\delta([\alpha]) 
	= \langle  [\alpha], \widetilde{S}_\Omega(\Phi)\rangle +  \int_M(\int_{\chi_\delta}\alpha )\Omega,
	$ 
	for any   $\Phi \in \ahom_\delta^\Omega(H)$ 
	where $ \widetilde{S}_\Omega$ is the usual flux  invariant for smooth isotopies. 
\end{proposition}
\textbf{Proof}:  
The proof is a direct application of the definitions of $\rdelta(H, \alpha)$ (equation \ref{Functor3}) and of $T_H^\delta$. In particular, we have:
$T_H^\delta([\alpha]) = \int_M \rdelta(H, \alpha) \Omega.$
We also know that, by definition:
$$
\rdelta(H, \alpha)(p) = \left\{
\begin{array}{lccc}
	( \flux_\alpha(H)(1))(p) &  \text{ if }  &  H\in \iso(M)\\
	( \flux_\alpha(\Phi)(1))(p) +  \int_{\chi_\delta^p}\alpha, & \text{ if } & H \in (\pcal(\homeo(M),\text{Id})\setminus \iso(M)),
\end{array}
\right.  
$$
for $\Phi \in \ahom_\delta^\Omega(H)$. Thus,
$$T_H^\delta([\alpha]) =  \int_M \left( \flux_\alpha(\Phi)(1) +  \int_{\chi_\delta}\alpha \right) \Omega = \int_M  \flux_\alpha(\Phi)(1) \Omega + \int_M (\int_{\chi_\delta}\alpha ) \Omega. $$
The first term is the value of $\widetilde{S}_\Omega(\Phi)$ acting on the cohomology class of $\alpha$. The second term is just what we have in the statement of the proposition. $\Box$

\begin{lemma}\label{lem4} Fix $\delta \in ]0,r(g)[$.  
	Assume that $H, H'\in \pcal\homeo_0(M)$ 
	have the same endpoints. 
	If there exists $z\in M$ for which each of the orbits $\ocal_z^H$ and 
	$\ocal_z^{H'}$ is piece-wise smooth so that the 
	$2-$chain delimited by $\ocal_z^{H'}$ and $\ocal_z^{H}$ 
	are null-homologous, then  
	$  \widetilde{L}_\Omega^\delta(H)  =   \widetilde{L}_\Omega^\delta(H')$. 
\end{lemma}
\textbf{Proof}:  This follows from Corollary \ref{lem1} :  Let $h$ be the time-one map of 
$H$ and  for each closed  $1-$form $\alpha$, derive from Corollary \ref{lem1} that 
\begin{eqnarray}
	- \rdelta(H,\alpha)(z) +\int_M \mathbf  R_\delta(H,\alpha)\Omega  & =& \chi(h,\alpha)(z)
	=  - \rdelta(H',\alpha)(z) + \int_M \rdelta(H',\alpha)\Omega ,\nonumber
\end{eqnarray}
namely, 
\begin{eqnarray} \langle [\alpha], \widetilde{L}_\Omega^\delta(H) -  \widetilde{L}_\Omega^\delta(H')\rangle
	&: =&\int_M \rdelta(H,\alpha)\Omega -\int_M \rdelta(H',\alpha)\Omega\nonumber\\
	& = & \rdelta(H,\alpha)(z)    - \rdelta(H',\alpha)(z) \nonumber\\
	&=& \int_{\ocal_z^{H}}\alpha - \int_{\ocal_z^{H'}}\alpha = 0,
\end{eqnarray}
for all $[\alpha]\in H^{1}(M,\mathbb{R})$. The last step comes from the fact that the 2-chain delimited by $\ocal_z^H$ and $\ocal_z^{H'}$ is null-homologous. Thus, $  \widetilde{L}_\Omega^\delta(H)  =   \widetilde{L}_\Omega^\delta(H')$. $\Box$

\begin{lemma}\label{lem5}  Let $H\in \pcal\homeo_0(M)$  be a loop at the identity map, and fix $\delta \in ]0,r(g)[$. The following hold.   
	\begin{enumerate}
		\item\label{sf1} For each closed $1-$form
		$\alpha,$ the  function $x\longmapsto \rdelta(H,\alpha)(x)$ is constant  and agrees with  
		$\langle [\alpha], \widetilde{L}_\Omega^\delta(H)\rangle$. 
		\item\label{sf2} Then,  $\widetilde{L}_\Omega^\delta(H) = 0$,  if and only if, there exists 
		$z\in M$ for which the orbit $\ocal_z^{H} $ is a piece-wise smooth boundary.

	\end{enumerate}
\end{lemma}

\begin{proof} 
	The item \ref{sf1}  follows directly from Corollary \ref{lem1}, which states that $\|\alpha \|_{L^2}\chi(h,\alpha)(z) 
	= -\left(\rdelta (H,\alpha)(z) - \int_M \rdelta(H,\alpha)\, \Omega\right)$. If $H$ is a loop at the identity, then $h$ is the identity map. Therefore,  $\chi(h,\alpha)(z) = 0$, which means that $\rdelta (H,\alpha)(z) = \int_M \rdelta(H,\alpha)\, \Omega$ is constant, and, by definition, equals to $\langle [\alpha], \widetilde{L}_\Omega^\delta(H)\rangle$.  The second item also follows directly from Corollary \ref{lem1}.  If $\widetilde{L}_\Omega^\delta(H) = 0$, then $\rdelta(H,\alpha)(x) = 0$, which implies that the integral of $\alpha$ over the orbit of any point $z$ is zero, and the orbit can be seen as a piecewise smooth boundary. Conversely, if there is a point $z$ such that $\ocal_z^{H}$ is a piecewise smooth boundary, then the integral of any closed 1-form over the orbit will be zero, which implies that  $\rdelta(H,\alpha)(x) = 0$, therefore, $\widetilde{L}_\Omega^\delta(H) = 0$.
\end{proof}
	
	\subsection{On the kernel of  $\ker L _\Omega^\delta$ }
	With the present construction, let  $\widetilde{\pcal\homeo_0(M,\Omega)}$ denote the quotient space of $\pcal\homeo_0(M, \Omega)$ with respect 
	to the following equivalence relation : $H, H'\in \pcal\homeo_0(M,\Omega)$ are equivalent if and only if $H$ and $H'$ are homotopic in $\homeo_0(M,\Omega)$ relatively to fixed end points. 
	Let denote by $\pi_1 \left(\homeo_0(M,\Omega) \right)$ the space of all loops at the identity map in $\pcal\homeo_0(M,\Omega)$, and set 
	$\tilde \Gamma_\Omega^\delta:= \widetilde{L}_\Omega^\delta\left( \pi_1 \left(\homeo_0(M,\Omega) \right)\right).$ 
	Therefore, there is a group homomorphism 
	$L_\Omega^\delta: \homeo_0(M,\Omega) \rightarrow H^{n-1}(M,\mathbb{R})/\tilde \varGamma_\Omega^\delta,$ such that the following 
	diagram commutes,  
	$$\begin{array}{ccc}
		\widetilde{\homeo_0(M,\Omega)} & \stackrel{\stackrel{\sim}{L_\Omega^\delta}}{\longrightarrow} & H^{n-1}(M,\mathbb{R}) \\ 
		\tilde \pi_1\downarrow &     & \downarrow \tilde \pi_2 \\ \homeo_0(M,\Omega) &  \stackrel{L_\Omega^\delta}{\longrightarrow} &
		H^{n-1}(M,\mathbb{R})/\tilde \Gamma_\Omega^\delta\hspace{1cm}(II), 
	\end{array}$$
	where $\tilde \pi_i$, $i = 1, 2$ are projection maps.

	\begin{example}\label{EX2}
		Consider the torus $\mathbb{T}^2$, with coordinates $(\theta_1, \theta_2)$, where $\theta_1, \theta_2 \in [0, 2\pi]$, and equipped with the flat Riemannian metric $g_0 = d\theta_1^2 + d\theta_2^2$. Let $\Omega = d\theta_1 \wedge d\theta_2$ be the standard volume form on $\mathbb{T}^2$. Let $\epsilon > 0$ be a small parameter, and let $\mu : [0, 2\pi] \rightarrow [0, +\infty[$ be a continuous function with the following properties:
		\begin{itemize}
			\item $\mu(\theta_1) = 0$ for $\theta_1 \in [2\pi - \epsilon, 2\pi]$.
			\item $\mu$ is not differentiable at $\theta_1 = 0$, having a sharp corner.
			\item $\mu'(\theta_1) < 0$ for $\theta_1 \in ]0, 2\pi - \epsilon[$, and is smooth.
		\end{itemize}
		
		Define the mapping $R_{\mu}(\theta_1, \theta_2) = (\theta_1, \theta_2 + \mu(\theta_1))$. This map is a homeomorphism of $\mathbb{T}^2$, and because it preserves the volume form $\Omega$, it is a volume-preserving homeomorphism, $R_{\mu} \in \text{Homeo}_0(\mathbb{T}^2, \Omega)$. However, due to the non-smoothness of $\mu$ at $\theta_1=0$, the map is not differentiable at $\theta_1=0$.
		
		\begin{tikzpicture}[scale=0.9]
			\begin{scope}[shift={(-4cm,0)},rotate=-20]
				\draw[thick,fill=gray!20] (0,0) ellipse (1.2cm and 0.4cm);
				\draw[thick,fill=gray!20] (0,-1.2cm) ellipse (1.2cm and 0.4cm);
				\draw[thick] (1.2cm,0) arc (0:180:1.2cm and 0.4cm);
				\draw[thick] (1.2cm,-1.2cm) arc (0:180:1.2cm and 0.4cm);
				\draw[thick,dashed] (-1.2cm,0) arc (180:360:1.2cm and 0.4cm);
				\draw[thick,dashed] (-1.2cm,-1.2cm) arc (180:360:1.2cm and 0.4cm);
				\draw[thick,fill=gray!30] (0,-1.2cm) ellipse (0.8cm and 0.32cm);
				\draw[thick,fill=gray!30] (0,-1.2cm) ellipse (0.76cm and 0.3cm);
				
				\node[above,rotate=20] at (2.2cm,0) {$\theta_1$};
				\node[above,rotate=20] at (0.4cm,-0.8cm) {$\theta_2$};
				\draw[thick, red]  (1.2cm,-0.15cm) arc (0:180:1.2cm and 0.3cm) ;
				\draw[thick, blue] (0.5cm,-1.2cm) arc (0:180:0.5cm and 0.15cm) ;
				\node[above,rotate=-20] at (0,-1.6cm) {Tore};
				
				\draw[->, thick, shorten >=0.1cm] (0, -0.3cm) to [out=-10,in=180] (-1.5cm, -1.4cm) ;
			\end{scope}
			
			\begin{scope}[shift={(3cm,0)}]
				\draw[->] (-0.5cm,0) -- (6cm,0) node[right] {$\theta_1$};
				\draw[->] (0,-1.2cm) -- (0,4cm) node[above] {$\mu(\theta_1)$};
				
				\draw[domain=0:5.8,samples=50,thick] plot (\x,{1.2*exp(-0.2*(\x))}); 
				\draw[domain=0:5.8,samples=50,thick,dashed,red] plot (\x,{1.2*exp(-0.2*(\x)) -0.1*sin(3*deg(\x))}); 
				\draw[dashed] (5.2cm,0) -- (5.2cm, -0.1cm);
				\node[below] at (5.2cm,-0.2cm) {$2\pi$};
				\draw[->, thick] (0.8cm, 0.2cm) to [out=90,in=-180] (0.4cm, 2.2cm) ;
				
				\node[above,align=left] at (0,2.8cm) {Graphe de $\mu$};
			\end{scope}
			\draw[->, thick, shorten >=0.1cm] (-1cm,-1.4cm) to [out=0,in=180] (0.5cm, 1cm) ;
			\node at (-1.5cm, 1.2cm)  {$R_{\mu}:(\theta_1,\theta_2)\mapsto (\theta_1, \theta_2+\mu(\theta_1))$};
		\end{tikzpicture}
		
		The goal is to compute the extended flux invariant $\stackrel{\sim}{L_\Omega^\delta}$ for the continuous path from the identity to $R_\mu$ defined by $R_\mu^t(\theta_1,\theta_2) = (\theta_1, \theta_2 + \mu(t\theta_1))$. This path is a natural choice for connecting the identity to $R_\mu$ because it can be seen as the straight line path in the space of maps. However, the usual flux invariant $\widetilde{S_\Omega}$ is not directly defined for this path because the maps $R_\mu^t$ are not differentiable. To address this issue, we will approximate the function $\mu$ by a smooth function $\nu : [0, 2\pi] \rightarrow [0, +\infty[$, that fixes zero, such that $|\nu - \mu|_{C^0} < \delta$.
		
		\begin{tikzpicture}[scale=1.2]
			\draw[->] (-0.5,0) -- (7,0) node[right] {$\theta_1$};
			\draw[->] (0,-0.5) -- (0,4) node[above] {$\mu(\theta_1), \nu(\theta_1)$};
			
			\newcommand{\muFunc}[1]{
				ifthenelse(#1<1.5, {1.2 - 0.4 * #1}, {1.2*exp(-0.5*(#1-1.5))})
			}
			
			\newcommand{\nuFunc}[1]{
				1.2 * exp(-0.3*#1) * (1 + 0.2*sin(deg(4*#1)))
			}
			\newcommand{\deltaValue}{0.3}
			
			\draw[thick,blue, domain=0:6.28, samples=100] plot (\x, {\muFunc{\x}});
			
			\draw[thick,red, dashed, domain=0:6.28, samples=100] plot (\x, {\nuFunc{\x}});
			
			\draw[fill=gray!20, opacity=0.5, domain=0:6.28, samples=100] plot (\x,{\muFunc{\x} - \deltaValue}) -- plot [domain=6.28:0, samples=100](\x, {\muFunc{\x} + \deltaValue}) --cycle;
			
			\node[above right] at (6.28, 0) {$2\pi$};
			\node[above left] at (0,0) {$0$};
			\node[right] at (6, 2) {$\mu(\theta_1)$};
			\node[right] at (6, 1.5) {$\nu(\theta_1)$};
			
			\draw[<->] (2, {(\muFunc{2} - \deltaValue)}) -- (2, {(\muFunc{2} + \deltaValue)}) node[midway, left] {$\delta$};
			
			\node[align=left, below] at (6, -0.5) {The smooth function $\nu$ is close to $\mu$, in such a way that $ |\nu - \mu|_{C^0} \le \delta$.};
		\end{tikzpicture}
		
		The parameter $\delta$ dictates how close the smooth approximation is. Assume that $\nu$ has the same support as $\mu$. The path $\{R_\nu^t\}: t \mapsto (\theta_1, \theta_2 + \nu(t\theta_1))$ is a smooth isotopy (a path of diffeomorphisms) and belongs to $\ahom_\delta^\Omega(\{R_\mu^t\})$, which is the set of all smooth isotopies homotopic relatively to fixed endpoint to an approximate $\Theta$ of $\{R_\mu^t\}$ such that $d_{C^0}(\Theta, \{R_\mu^t\})< \delta$. Then, by definition, we have
		\[
		\stackrel{\sim}{L_\Omega^\delta}(\{R_{\mu}^t\}) =   \widetilde{S_\Omega}(\{R_{\nu}^t\}) + T(R_{\mu}^1, R_{\nu}^1),
		\]
		where $T(R_{\mu}^1, R_{\nu}^1)$ is an element in the first de Rham group of $\mathbb{T}^2$, such that
		\[
		\langle[\alpha], T(R_{\mu}^1, R_{\nu}^1)\rangle = \int_{\mathbb{T}^2}\left(\int_{\chi_\delta}\alpha\right)\Omega,
		\]
		for all closed 1-forms $\alpha$ on $\mathbb{T}^2$, and $\chi_\delta$ is a minimizing geodesic (in the flat metric) from $R_{\mu}^1(x)$ to $R_{\nu}^1(x)$ for all $x \in \mathbb{T}^2$. We will calculate each of these terms in turn.
		
		\textbf{Calculation of the standard flux:}
		We compute the flux of the smooth isotopy $\{R_{\nu}^t\}$. The time derivative of the path is:
		\[
		\frac{d}{dt} R_\nu^t (\theta_1, \theta_2) = \frac{d}{dt} (\theta_1, \theta_2 + \nu(t\theta_1)) = (0, \nu'(t\theta_1)\theta_1).
		\]
		
		The interior product of this vector with the volume form $\Omega$ is:
		\[
		\iota_{\dot{R_\nu^t}} \Omega  = \iota_{(0, \nu'(t\theta_1)\theta_1)}(d\theta_1 \wedge d\theta_2) = \nu'(t\theta_1)\theta_1 d\theta_1.
		\]
		Therefore, the flux is:
		\begin{align*}
			\widetilde{S_\Omega}(\{R_{\nu}^t\}) &= \left[\int_0^1 \iota_{\dot{R_\nu^t}} \Omega \, dt\right]  \\
			&=  [\nu(\theta_1)d\theta_1].
		\end{align*}
		In particular, when this map acts on the class of a $1$-form $[\alpha]$, it produces the value:
		\[
		\langle [\alpha], \widetilde{S_\Omega}(\{R_{\nu}^t\})\rangle =  \int_{\mathbb{T}^2} \nu(\theta_1)\alpha\left(\frac{\partial}{\partial\theta_2}\right)  d\theta_1 \wedge d\theta_2.
		\]
		
		\textbf{Calculation of the correction term $T(R_{\mu}^1, R_{\nu}^1)$:}
		The correction term $T(R_{\mu}^1, R_{\nu}^1)$ measures the difference between the orbits of $R_\mu^1$ and $R_\nu^1$ as seen by closed 1-forms. For a given point $(\theta_1, \theta_2) \in \mathbb{T}^2$, $R_{\mu}^1(\theta_1, \theta_2) = (\theta_1, \theta_2 + \mu(\theta_1))$, and $R_{\nu}^1(\theta_1, \theta_2) = (\theta_1, \theta_2 + \nu(\theta_1))$. The minimizing geodesic $\chi_\delta$ between $R_{\mu}^1(\theta_1, \theta_2)$ and $R_{\nu}^1(\theta_1, \theta_2)$ is simply a vertical line segment (in the flat metric) in the torus. This segment starts at $(\theta_1, \theta_2 + \mu(\theta_1))$ and ends at $(\theta_1, \theta_2 + \nu(\theta_1))$. Thus, we can parametrize this line segment as:
		\[
		\chi_\delta(s) = (\theta_1, \theta_2 +  \mu(\theta_1) +  s (\nu(\theta_1) - \mu(\theta_1))).
		\]
		
		Let $\alpha$ be a closed 1-form. Since we are in the flat torus, we can write $\alpha = a d\theta_1 + b d\theta_2$, for some constants $a$ and $b$. Thus,
		\begin{align*}
			\int_{\chi_\delta} \alpha &= \int_0^1 \left(a \frac{d \theta_1}{ds} + b \frac{d (\theta_2 + \mu(\theta_1) +  s (\nu(\theta_1) - \mu(\theta_1)))}{ds}\right) ds \\
			&= \int_0^1 (a \cdot 0 + b (\nu(\theta_1) - \mu(\theta_1))) ds \\
			&= b (\nu(\theta_1) - \mu(\theta_1)).
		\end{align*}
		Therefore,
		\[
		\langle [\alpha], T(R_{\mu}^1, R_{\nu}^1)\rangle = \int_{\mathbb{T}^2} \left(\int_{\chi_\delta}\alpha\right)\Omega = \int_{\mathbb{T}^2} b (\nu(\theta_1) - \mu(\theta_1)) d\theta_1 d\theta_2.
		\]
		If we consider $\alpha = d\theta_2$, then we have that $a=0$ and $b=1$, thus:
		\[
		\langle [d\theta_2], T(R_{\mu}^1, R_{\nu}^1)\rangle =  \int_{\mathbb{T}^2} (\nu(\theta_1) - \mu(\theta_1)) d\theta_1 d\theta_2 =  2\pi \int_0^{2\pi} (\nu(\theta_1) - \mu(\theta_1)) d\theta_1.
		\]
		If we consider $\alpha = d\theta_1$, then $a=1$ and $b=0$, so:
		\[
		\langle [d\theta_1], T(R_{\mu}^1, R_{\nu}^1)\rangle = 0.
		\]
		So, we can write $T(R_{\mu}^1, R_{\nu}^1)$ as:
		\[
		T(R_{\mu}^1, R_{\nu}^1) = \left( 2\pi \int_0^{2\pi} (\nu(\theta_1) - \mu(\theta_1)) d\theta_1\right)  [d\theta_2].
		\]
		
		\textbf{Extended flux calculation:}
		Thus,
		\begin{align*}
			\stackrel{\sim}{L_\Omega^\delta}(\{R_{\mu}^t\}) &= \widetilde{S_\Omega}(\{R_{\nu}^t\}) + T(R_{\mu}^1, R_{\nu}^1) \\
			&= [\nu(\theta_1)d\theta_1]  + \left( 2\pi \int_0^{2\pi} (\nu(\theta_1) - \mu(\theta_1)) d\theta_1\right) [d\theta_2].
		\end{align*}

	{\bf Calculation of the correction term $T(R_{\mu}^1, R_{\nu}^1)$:}  The correction term $T(R_{\mu}^1, R_{\nu}^1)$ measures the difference between the orbits of $R_\mu^1$ and $R_\nu^1$ as seen by closed 1-forms. For a given point $(\theta_1, \theta_2) \in \mathbb{T}^2$,  $R_{\mu}^1(\theta_1, \theta_2) = (\theta_1, \theta_2 + \mu(\theta_1))$, and  $R_{\nu}^1(\theta_1, \theta_2) = (\theta_1, \theta_2 + \nu(\theta_1))$. The minimizing geodesic $\chi_\delta$ between $R_{\mu}^1(\theta_1, \theta_2)$ and $R_{\nu}^1(\theta_1, \theta_2)$ is simply a vertical line segment (in the flat metric) in the torus. This segment starts at $ (\theta_1, \theta_2 + \mu(\theta_1))$ and ends at $(\theta_1, \theta_2 + \nu(\theta_1))$. Thus, we can parametrize this line segment as:
	\[
	\chi_\delta(s) = (\theta_1, \theta_2 +  \mu(\theta_1) +  s (\nu(\theta_1) - \mu(\theta_1)) ).
	\]
	
	Let $\alpha$ be a closed 1-form. Since we are in the flat torus, we can write $\alpha = a d\theta_1 + b d\theta_2$, for some constants $a$ and $b$. Thus,
	\begin{align*}
		\int_{\chi_\delta} \alpha &= \int_0^1 (a \frac{d \theta_1}{ds} + b \frac{d (\theta_2 + \mu(\theta_1) +  s (\nu(\theta_1) - \mu(\theta_1)) )}{ds})ds \\
		&= \int_0^1 (a \cdot 0 + b (\nu(\theta_1) - \mu(\theta_1))) ds \\
		&= b (\nu(\theta_1) - \mu(\theta_1)).
	\end{align*}
	Therefore,
	\[
	\langle [\alpha], T(R_{\mu}^1, R_{\nu}^1)\rangle = \int_{\mathbb{T}^2} \left(\int_{\chi_\delta}\alpha\right)\Omega = \int_{\mathbb{T}^2} b (\nu(\theta_1) - \mu(\theta_1)) d\theta_1 d\theta_2.
	\]
	If we consider $\alpha = d\theta_2$, then we have that $a=0$ and $b=1$, thus:
	\[
	\langle [d\theta_2], T(R_{\mu}^1, R_{\nu}^1)\rangle =  \int_{\mathbb{T}^2} (\nu(\theta_1) - \mu(\theta_1)) d\theta_1 d\theta_2 =  2\pi \int_0^{2\pi} (\nu(\theta_1) - \mu(\theta_1)) d\theta_1.
	\]
	If we consider $\alpha = d\theta_1$, then $a=1$ and $b=0$, so:
$
	\langle [d\theta_1], T(R_{\mu}^1, R_{\nu}^1)\rangle = 0.
	$
	So, we can write $T(R_{\mu}^1, R_{\nu}^1)$ as:
	$
	T(R_{\mu}^1, R_{\nu}^1) = \left( 2\pi \int_0^{2\pi} (\nu(\theta_1) - \mu(\theta_1)) d\theta_1\right)  [d\theta_2].
	$\\
	{\bf Extended flux calculation:} Thus,
	\begin{align*}
		\stackrel{\sim}{L_\Omega^\delta}(\{R_{\mu}^t\}) &= \widetilde{S_\Omega}(\{R_{\nu}^t\}) + T(R_{\mu}^1, R_{\nu}^1) \\
		&= [\nu(\theta_1)d\theta_1]  + \left( 2\pi \int_0^{2\pi} (\nu(\theta_1) - \mu(\theta_1)) d\theta_1\right) [d\theta_2]. 
	\end{align*}
	\begin{center}
	\begin{tikzpicture}[scale=1.1]
		
		\draw[thick, blue, decorate, decoration={snake,amplitude=0.1cm,segment length=0.2cm}] (0,0) to (3,0);
		\node[above] at (1.5,0.2) {$\widetilde{S}_\Omega(\Phi)$};
		
		\draw[->, thick, dashed, red] (3,0) to[out=0, in=180] (4.5,-1);
		\node[below right] at (4.5,-1) {$T(R_{\mu}^1, R_{\nu}^1)$};
		
		\draw[->, thick, black] (0,0) to[out=-30,in=180] (4.5,-1);
		\node[above] at (2,-0.8) {$\stackrel{\sim}{L_\Omega^\delta}(H)$};
		
		\node[below,align=left] at (3,-2) {Extended flux as smooth flux with an added correction term : $\stackrel{\sim}{L_\Omega^\delta}(H) = \widetilde{S_\Omega}(\Phi) + T(R_{\mu}^1, R_{\nu}^1)$.};

	\end{tikzpicture}
	\end{center}
	This result demonstrates that the extended flux indeed provides a value for the flux of a $C^0$ path, even though it is not differentiable.\\

	{\bf  Fathi's mass flow:}
	The Fathi's mass flow of $\{R_\mu^t\}$ is given by:
	\[
	\widetilde{\theta}(\{R_{\mu}^t\})(f) = \left\langle [f^\ast\sigma], T(R_{\mu}^1, R_{\nu}^1) + [\nu(\theta_1) d\theta_1]\right\rangle,
	\]
	for all continuous mappings $f$ from $\mathbb{T}^2$ to the circle $\mathbb{S}^1$, where $\sigma$ is the volume form on $\mathbb{S}^1$. That is, 
	$$ \widetilde{\theta}(\{R_{\mu}^t\})(f) =  \widetilde{\theta}(\{R_{\nu}^t\})(f) + \left\langle [f^\ast\sigma], T(R_{\mu}^1, R_{\nu}^1) \right\rangle, $$
	for all continuous mappings $f$ from $\mathbb{T}^2$ to the circle $\mathbb{S}^1$, where $\sigma$ is the volume form on $\mathbb{S}^1$. The first term corresponds to the usual smooth flux (in the symplectic case), and the term involving $T(R_{\mu}^1, R_{\nu}^1) $ represents the necessary correction needed due to the lack of smoothness in the path.  This is where the non-smoothness of $\{R_{\mu}^t\}$ is addressed. The limit ensures that the integral is properly defined even when the path is not differentiable at all times. The equation shows how the Poincaré duality theorem works even in the presence of non-smooth paths. In this specific example, it is known that the Fathi's mass flow coincides with the image of the path by the extended flux. Since the flux group  $\Gamma_\Omega$ of $\mathbb{T}^2$ is isomorphic to $\mathbb{Z}\oplus\mathbb{Z}$, we have that $\stackrel{\sim}{L_\Omega^\delta}(\{R_{\mu}^t\})$ belongs to  $\mathbb{Z}\oplus\mathbb{Z}$. This shows that we can use the extended flux to obtain a well defined invariant for homeomorphisms.
\end{example}
	\begin{proposition}\label{pro4}  
	Let $H := \{h_t\}\in \pcal\homeo_0(M, \Omega)$. Then, $ h_1\in \ker L_\Omega$ if and only if $\stackrel{\sim}{L_\Omega^\delta}(H)\in\tilde  \Gamma_\Omega^\delta$.
\end{proposition}
	
	
We equip the subgroup $\tilde \Gamma_\Omega\subset  H^{n-1}(M,\mathbb{R})$, with the natural topology  arising from the vector 
space structure of $ H^{n-1}(M,\mathbb{R})$.  
	\begin{lemma}\label{Disc-0}
		The group $\tilde{\Gamma}_\Omega^\delta$ is discrete.
	\end{lemma}
	
	\begin{proof}
		This follows by adapting the proof of Theorem 4.6 of \cite{HLGF}.
		Let $H\in \pi_1 \left(\text{Homeo}_0(M, \Omega)\right)$.
		The Poincar\'e pairing $\langle\cdot,\cdot\rangle : H^1(M,\mathbb{R})
		\times H^{(n-1)}(M,\mathbb{R}) \longrightarrow \mathbb{R},$ being
		continuous, let $\mu_0$ denote any
		positive constant such that
		\begin{equation}\label{Disc1}
			\big|\langle A, B \rangle\big|\leqslant \mu_0\|A\|_{L^2}\|B\|_{L^2},
		\end{equation}
		for
		all $(A, B)\in H^1(M,\mathbb{R})
		\times H^{(n-1)}(M,\mathbb{R})$. Let $\alpha$ be any closed $1-$form such that its de Rham
		cohomology class
		$[\alpha]$ belongs to the closed unit ball $\overline{\mathbb{B}^1(0,1)}\subset H^1(M,\mathbb{R}) $,
		and consider the continuous linear mapping
		$L_\alpha : H^{(n-1)}(M,\mathbb{R}) \longrightarrow \mathbb{R}, B\longmapsto \langle [\alpha], B \rangle$. It
		follows from (\ref{Disc1}) that $| L_\alpha(B)| = |\langle [\alpha], B \rangle| \leqslant
		\mu_0\| B \|_{L^2}$, for all
		$B\in H^{(n-1)}(M,\mathbb{R})$ because $ [\alpha] \in \overline{\mathbb{B}^1(0,1)}$.
		
		Therefore, we have two possibilities: If $\mathcal{O}_{z}^{H}$ is null-homologous, then
		Lemma \ref{lem5} implies that $\widetilde{L}_\Omega^\delta(H) = 0$. If $\mathcal{O}_{z}^{H}$ is not homologically trivial, then assume
		$H^{(n-1)}(M, \mathbb{R})$ equipped with the $L^2-$Hodge norm $\|\cdot\|_{L^2}$. Since $\widetilde{L}_\Omega^\delta(H)$ is considered as arbitrarily small in $H^{(n-1)}(M,\mathbb{R})$, let
		$l$ be any arbitrary positive integer, and assume that
		\begin{equation}\label{INT}
			\|\widetilde{L}_\Omega^\delta(H) \|_{L^2}< \frac{1}{2l\mu_0 },
		\end{equation}
		where $\mu_0$ is the constant defined in (\ref{Disc1}). On the other hand, we proceed to choose a closed 1-form $\alpha_0$ on $M$ with $[\alpha_0] \in \overline{\mathbb{B}^1(0,1)}$ such that $\mathcal{R}(H, \alpha_0) \neq 0$. Since $\mathcal{O}_z^H$ is not homologically trivial, the homology class represented by this loop is non-zero in $H_1(M, \mathbb{R})$.  By the fundamental connection between homology and integration, there exists a closed 1-form $\alpha_0$ such that
		\[
		\oint_{\mathcal{O}_z^H} \alpha_0 \neq 0.
		\]
		Moreover, we can choose $\alpha_0$ such that $[\alpha_0] \in \overline{\mathbb{B}^1(0,1)}$.
		Now, we have that there exists $\alpha_0$ such that $\mathcal{R}(H, \alpha_0)(z) = \int_{\mathcal{O}_z^H} \alpha_0  \neq 0$.
		Using Lemma \ref{lem5} together with the inequality (\ref{INT}), we derive that
		$$0< \left| \mathcal{R}(H,\alpha_0) \right|
		= \big|\langle [\alpha_0], \widetilde{L}_\Omega^\delta(H)\rangle\big|
		\leqslant \mu_0\|\widetilde{L}_\Omega^\delta(H) \|_{L^2} < \frac{1}{l},$$
		for all positive integer $l$. This is a contradiction because the last term on right-hand side
		of the above inequalities tends to zero as $l$ tends to infinity.
		Therefore, $\widetilde{L}_\Omega^\delta(H)$ must be trivial. This shows that any arbitrary small element in $ \tilde \Gamma_\Omega^\delta$ is necessary trivial, i.e., the
		trivial element in $ \tilde \Gamma_\Omega^\delta$ has an isolated open neighborhood and since translation maps are homeomorphisms, then the latter isolated open neighborhood of trivial element in $ \tilde \Gamma_\Omega^\delta$ generates an isolated open neighborhood of any other element in $ \tilde \Gamma_\Omega^\delta$.
	\end{proof}
Lemma \ref{Disc-0} means that the types of "twists" that can be produced by these transformations are quantized. This limits the possible long-term behaviors of the system.\\

	{\it Question(A):} From the construction, we have $ \Gamma_\Omega\subset \tilde \Gamma_\Omega^\delta$. But, what about the converse inclusion:  $ \tilde\Gamma_\Omega^\delta\subset \Gamma_\Omega$?\\
	The above question raises the problem of $C^0-$rigidity of the usual flux group $\Gamma_\Omega$. 
	\begin{itemize}
	\item 	If $ \tilde\Gamma^\delta_{\Omega} \subset \Gamma_{\Omega} $: this implies that all "perturbed" fluxes in a small neighborhood around the identity can be realized as true smooth volume-preserving paths  fluxes. This scenario suggests a form of $C^0$-rigidity, indicating that the space of vanishing flux isotopies is stable under small $C^0$ perturbations—even those that only approximately preserve the volume form.
	\item Conversely, if $\tilde \Gamma^{\delta}_{\Omega} \nsubseteq \Gamma_{\Omega} $: 
		it indicates the existence of small perturbations around the identity that can still preserve the volume form but yield "perturbed" fluxes that cannot be created via a  smooth volume-preserving paths. This suggests a breakdown of $C^0$-rigidity.
	\end{itemize}
	 By a result of Smale, the inclusion $Diff(M) \hookrightarrow \homeo(M)$ induces an isomorphism on fundamental groups: $\pi_1(Diff(M)) \cong \pi_1(\homeo(M))$.  
	 This means that any loop in $\pi_1(\homeo(M))$ can be represented by a loop in $\pi_1(Diff(M))$ \cite{Hirs76}.  On account of the locally contractibility of $ Diff^{\Omega,\infty}_0(M)$ due to Thurston \cite{TH}, we prove the following weak version of the above Smale result for volume-preserving transformations. 

		\begin{corollary}\label{Den-1}
			The group $ \pi_1(Diff^{\Omega,\infty}_0(M))$ is $C^0-$dense in $\pi_1(\homeo_0(M,\Omega))$. 
		\end{corollary}
		\begin{proof}
			Let $H\in \pi_1(\homeo_0(M,\Omega))$. For each positive integer $n$ such that $1/n \leqslant \delta$, select $\Phi_n = (\phi^t_n)\in \ahom_\delta^\Omega(H)$ such that 
			\( d_{C^0}(\Phi_n, H )\leq 1/n\). Since the time-one map of $H$ is the identity map, then for n sufficiently large, the time-one map $\phi_n^1$ belongs to a small open ball $\mathcal{U}_{id} $ centered  at the identity map in $\homeo_0(M,\Omega)$ of radius $1/n$ with respect to the  $C^0-$metric (if necessary, increase $n$).  On account of the locally path-connectedness of $Homeo(M, \Omega)$, we can find a volume-preserving isotopy $\Theta_n$ in $\mathcal{U}_{id} $ from $\phi_n^1$ to identity map for $n$ sufficiently large.
			Apply Theorem \ref{Theo-4} to approximate $\Theta_n$ by a continuous isotopy $\bar{\varPsi}_n$ in $Diff^{\Omega,\infty}(M)$ (relatively to fixed endpoints) such that $ d_{C^0}(\bar{\varPsi}_n, \Theta_n)\leqslant 1/n$. By Thurston's local contractibility result for $Diff^{\Omega,\infty}_0(M)$, we can approximate $\Theta_n$ by a continuous, volume-preserving isotopy $\Psi_n$ in $Diff^{\Omega,\infty}_0(M)$, keeping the endpoints fixed, such that $d_{C^0}(\Psi_n, \Theta_n) \leqslant 1/n$.
			Concatenate the path $\Phi_n$ with the reverse of the path $\Psi_n$ to create a loop $\Upsilon^n$ in $Diff^{\Omega,\infty}_0(M)$. We have
			\begin{align*}
				d_{C^0}(\Upsilon^n, H) &\leqslant d_{C^0}(\Phi_n, H) + d_{C^0}(Id, \Psi_n) \\
				&\leqslant d_{C^0}(\Phi_n, H) + d_{C^0}(\Theta_n, \Psi_n) + d_{C^0}(Id, \Theta_n)\\
				&\leqslant \frac{1}{n} + \frac{1}{n} + \frac{1}{n} = \frac{3}{n},
			\end{align*}
			for $n$ sufficiently large. Here, $Id$ denotes the constant map at the identity, and we use the fact that $\Theta_n$ is contained in $\mathcal{U}_{id}$, which has radius $1/n$.

		\end{proof}

		\begin{theorem}\label{Theo}($(C^0,\delta)-$Rigidity)
			The groups $\tilde \Gamma_\Omega^\delta$ and $\Gamma_\Omega$ coincide. That is,
			the group $\Gamma_\Omega$ is $(C^0,\delta)-$rigid.
		\end{theorem}
		
		\begin{proof}
			We want to show that $\tilde{\Gamma}_\Omega^\delta = \Gamma_\Omega$. We know that $\Gamma_\Omega \subseteq \tilde{\Gamma}_\Omega^\delta$. The goal is to prove the reverse inclusion: $\tilde{\Gamma}_\Omega^\delta \subseteq \Gamma_\Omega$. Let $B \in \tilde{\Gamma}_\Omega^\delta$. By definition, there exists $H \in \pi_1(\text{Homeo}_0(M, \Omega))$ such that $B = \widetilde{L}_\Omega^\delta(H)$. Thus, we're starting with the perturbed, extended flux of a loop in the space of volume-preserving homeomorphisms. Corollary \ref{Den-1} states that $\pi_1(\text{Diff}_0^{\Omega, \infty}(M))$ is $C^0$-dense in $\pi_1(\text{Homeo}_0(M, \Omega))$. This means that there exists a sequence of loops $\Upsilon^n \in \pi_1(\text{Diff}_0^{\Omega, \infty}(M))$ such that $d_{C^0}(\Upsilon^n, H) \to 0$ as $n \to \infty$. The sequence $\Upsilon_n$ is a sequence of volume preserving isotopies. Since $\Upsilon^n$ is a smooth volume-preserving isotopy, $\widetilde{L}_\Omega^\delta(\Upsilon^n) = \widetilde{S}_\Omega(\Upsilon^n)$, where $\widetilde{S}_\Omega$ is the usual (smooth) flux homomorphism. That is, the extended flux applied to a smooth loop is the same as the usual smooth flux. Therefore, $\widetilde{S}_\Omega(\Upsilon^n) \in \Gamma_\Omega$ for all $n$. We need to show that $\widetilde{L}_\Omega^\delta(\Upsilon^n)$ converges to $\widetilde{L}_\Omega^\delta(H)$ in $H^{n-1}(M, \mathbb{R})$. Consider the sequence $\{\widetilde{L}_\Omega^\delta(\Upsilon^n)\}$.  
			Given the $C^0$ convergence, $d_{C^0}(\Upsilon^n, H) \to 0$, as $n \rightarrow\infty$, we will have that $\Upsilon^n$ uniformly approximates $H$. Then the mapping $\mathbf{R}_\delta$ must satisfies:
			$$\lim_{n \rightarrow \infty} \int_M \mathbf{R}_\delta( \Upsilon^n,\alpha )\Omega = \int_M \mathbf{R}_\delta(H, \alpha)\Omega,$$
			for each closed $1-$form $\alpha$. This is equivalent to 
			$$\lim_{n \rightarrow \infty}  \langle [\alpha], \widetilde{L}_\Omega^\delta(\Upsilon^n) \rangle  = \langle [\alpha], \widetilde{L}_\Omega^\delta(H) \rangle,$$
			or each closed $1-$form $\alpha$. That is, $ \|\widetilde{L}_\Omega^\delta(\Upsilon^n)  -\widetilde{L}_\Omega^\delta(H)\|_{L^2}\rightarrow0, n\rightarrow\infty$.  On the other hand, 
			\begin{align*}
				d(\stackrel{\sim}{L_\Omega^\delta}(H), \Gamma_\Omega) & \leq 
				\|\stackrel{\sim}{L_\Omega^\delta}(H) -  \widetilde{S}_\Omega(\Upsilon^n)\|_{L^2} \\
				&=  \|\stackrel{\sim}{L_\Omega^\delta}(H) - \stackrel{\sim}{L_\Omega^\delta}(\Upsilon^n) \|_{L^2},
			\end{align*}
			for all $n$. 
			Since $ \|\widetilde{L}_\Omega^\delta(\Upsilon^n)  -\widetilde{L}_\Omega^\delta(H)\|_{L^2}\rightarrow0, n\rightarrow\infty$, then the discreteness of $\Gamma_\Omega$, together with the fact that $ \|\widetilde{L}_\Omega^\delta(\Upsilon^n)  -\widetilde{L}_\Omega^\delta(H)\|_{L^2}$ can be made arbitrarily  small implies that $\stackrel{\sim}{L_\Omega^\delta}(H) \in \Gamma_\Omega$. Therefore, $\tilde \Gamma_\Omega^\delta \subseteq \Gamma_\Omega$.
		\end{proof}

	\textbf{Information about $H_1(\text{Homeo}_0(M, \Omega); \mathbb{Z})$:} Because the extended flux provides a homomorphism from $\pi_1(\text{Homeo}_0(M, \Omega))$ to $H^{n-1}(M, \mathbb{R})$, it gives a lower bound on the rank of $H_1(\text{Homeo}_0(M, \Omega); \mathbb{Z})$.  In particular, if $\Gamma_\Omega$ is non-trivial, then $H_1(\text{Homeo}_0(M, \Omega); \mathbb{Z})$ must also be non-trivial.
	\begin{lemma}\label{lem9} 
		Let $H\in \pcal\homeo_0(M,\Omega)$ be a loop. If there exists $z\in M$ such that the $1-$cycle $\ocal_{z}^{H}$ is piecewise smooth, then it is null-homologous.
	\end{lemma}
	
	\begin{lemma}\label{lem10}  The group $\ker L_\Omega^\delta$ is a normal subgroup of $\homeo_0(M,\Omega) $.
	\end{lemma}
	
\begin{lemma}\label{lem13}  
	Any topological volume-preserving isotopy in \( \ker L_\Omega^\delta \) is a topological vanishing-flux isotopy.  
\end{lemma}  

\begin{proof}  
	Let \( H \) be any topological volume-preserving isotopy in \( \ker L_\Omega^\delta \). For each fixed \( t \), define the mapping \( Q_t: s \mapsto H^{st} \). Since \( Q_t(s) = H^{st} \in \ker L_\Omega^\delta \), it follows that \( \widetilde{L_\Omega^\delta}(Q_t) \in \Gamma_\Omega \). Thus, we have a continuous map   
	$ 
	Q: [0,1] \longrightarrow \Gamma_\Omega, \quad t \mapsto \widetilde{L_\Omega^\delta}(Q_t).  
	$  
	Since \( \Gamma_\Omega \) is discrete and \( [0,1] \) is connected, the map \( Q \) must be constant. Therefore, we have   
	$
	\widetilde{L_\Omega^\delta}(Q_t) = \widetilde{L_\Omega^\delta}(Q_0) = 0, \quad \forall t \in [0,1].  
	$ 
	In particular, this gives us   
	$
	\widetilde{L_\Omega^\delta}(H) = \widetilde{L_\Omega^\delta}(Q_1) = 0.   
	$  
\end{proof}  

	\subsection{Measure-preserving homeomorphisms}
	
	Denote by $\mu$ the Liouville measure induced by the volume form $\Omega$. Recall that $\homeo_0(M,\mu)$ 
	denotes the identity component 
	in the group of measure-preserving homeomorphisms $\homeo(M,\mu)$.
	Let $\widetilde {\homeo_0(M,\mu)}$ be the universal 
	covering of $\homeo_0(M,\mu)$. Let $[M,\mathbb{R}/\mathbb{Z}]$ be the space of all homotopy classes of continuous maps from $M$ onto $\mathbb{R}/\mathbb{Z} = \mathbb{T}$. 
	By identifying 
	$\mathbb{T}$ with $\mathbb{S}^1$, the group law on the space $\mathbb{S}^1$ is denoted additively. Due to Fathi \cite{Fathi}, the group $\homeo_0(M,\mu)$ is locally 
	contractible; hence, the quotient space of equivalent classes of homotopic paths isotopic to the identity with fixed endpoints 
	in $\homeo_0(M,\mu)$ 
	coincides 
	with the universal covering of $\homeo_0(M,\mu)$. For $[h] = [(h_t)]\in \widetilde {\homeo_0(M,\mu)}$, and a continuous map $f : M \rightarrow \mathbb{S}^1$, we lift the homotopy $fh_t - f : M \rightarrow \mathbb{S}^1$ to a map  $\overline{fh_t - f}$ from $M$ onto 
	$ \mathbb{R}$. Fathi proved that the integral $\int_M\overline{fh_t - f}d\mu$ depends only on the homotopy class $[h]$ of $(h_t)$ and the homotopy class 
	$\{f\}$ of $f$ in $[M,\mathbb{S}^1]\approx H^1(M,\mathbb{Z})$, and that the map 
	\[
	\widetilde{\mathfrak{F}}((h_t))(f):= \int_M\overline{fh_t - f}d\mu.
	\]
	defines a homomorphism 
	\[
	\widetilde{\mathfrak{F}} :\widetilde {\homeo_0(M,\mu)}\rightarrow \homo( H^1(M,\mathbb{Z}),\mathbb{R})\approx H_1(M,\mathbb{R}).
	\]
	Put
	$\Gamma(\mu) = \widetilde{\mathfrak F}(\pi_1(\homeo_0(M,\mu))),$ where $\pi_1(\homeo_0(M,\mu))$ is the first fundamental group of the topological 
	space $\homeo_0(M,\mu)$. The set $\Gamma(\mu)$ is in fact a discrete subgroup inside 
	$ H_1(M,\mathbb{R})$. The continuous epimorphism $\widetilde{\mathfrak{F}}$ induces a continuous epimorphism 
	$\mathcal{F}$ from $\homeo_0(M,\mu)$ onto $H_1(M,\mathbb{R})/\Gamma(\mu)$. It was shown in \cite{Fathi} that the flux for volume-preserving diffeomorphisms is the Poincaré dual of Fathi's mass flow. In the present paper, we extend this duality result to the $C^0$-closure of volume-preserving homeomorphisms.
	
	\begin{lemma}\label{lem8}  Let $H\in \pcal\homeo_0(M,\Omega)$. Then, the cohomology class 
		$\widetilde{L}_\Omega^\delta(H)$ is the Poincaré dual of the homology class $\widetilde{\theta}(H)$, 
		where $\widetilde{\theta} $ stands for the Fathi's mass flow homomorphism:
		\[
		\widetilde{\theta}(H)(f) = \langle \widetilde{L}_\Omega^\delta(H), f^\ast\sigma\rangle,
		\]
		for all continuous mappings $f$ from $M$ to the circle $\mathbb{S}^1$, where $\sigma$ is the volume form on $\mathbb{S}^1$. 
	\end{lemma}
	\section*{Application to finite energy homeomorphisms}
		The study of symplectic homeomorphisms and their energy properties has emerged as a central theme in symplectic topology and Hamiltonian dynamics.  In finite dimensions, the group of symplectic diffeomorphisms, and more recently, symplectic homeomorphisms, plays a role analogous to Lie groups in other geometric contexts.  Understanding the structure of these groups, particularly in terms of energy-related concepts like Hofer's geometry and flux theory \cite{Ban97}, is crucial for advancing our knowledge of symplectic manifolds and their dynamics.
	
	This paper studies finite energy symplectic homeomorphisms, focusing specifically on those with trivial flux on closed symplectic manifolds.  Our primary goal is to investigate whether, under certain topological conditions, a finite energy symplectic homeomorphism with trivial flux is necessarily a finite energy Hamiltonian homeomorphism \cite{cr2}.  This question is motivated by the desire to understand the relationship between symplectic and Hamiltonian dynamics at the topological level, particularly in the context of energy bounds.
	
	We concentrate our analysis on the 2-torus $(T^2, \omega)$, a fundamental example in symplectic geometry, while also establishing results that hold more generally for closed symplectic manifolds of Lefschetz type.  Our approach relies on a blend of approximation techniques, flux theory, and tools from Hofer-like geometry.  A cornerstone of our method is the use of Hodge decomposition to separate symplectic isotopies into harmonic and Hamiltonian components.  To control the energy  associated with these decompositions, we revisite and apply a boundedness result for two-parameter families of vector fields, leveraging Grönwall's inequality \cite{TKS-2}.
	
	A key contribution of this work is to demonstrate that on the 2-torus. This result sheds light on the rigidity properties of symplectic homeomorphisms with constraints on their flux and energy.  Furthermore, we establish a proposition concerning the loop property of harmonic flows on $T^2$, connecting the flux group to the periodicity of harmonic vector fields.\\

	Through this investigation, we aim to contribute to a deeper understanding of the interplay between symplectic topology, Hamiltonian dynamics, and the topological properties of homeomorphism groups in infinite-dimensional settings. 
		\subsubsection{Basics notions}
	
	Any symplectic isotopy  $\Phi=\{\phi^t\}$ is generated by a pair $(U,\mathcal{H})$, where $U=\{U^t\}$ is a smooth  time dependent function on $\Sigma_{g}\times [0,1]$ and $\mathcal{H}=\{\mathcal{H}^t\}$ is a smooth family of  $S-$forms or  harmonic forms. Therefore, the Hofer-like length of $\Phi=\{\phi^t\}$  is given by 
	
	\begin{equation}
		l^{(1,\infty)}_{\kappa, \mathcal{S}}(\Phi) 
		= \int_{0}^{1}\left ( osc(U^t) + \kappa\| \mathcal{H}^t\| _{L^2}\right )dt,
	\end{equation} 
	where $osc(\cdot)= \max(\cdot)-\min(\cdot)$,  $\kappa$ is a positive real number and $\|\cdot\|_{L^2}$ is the $L^2$-norm on $H^1(M,\mathbb R)$.
	The length $l^{(1,\infty)}_{\kappa, \mathcal{S}}$ induces  metrics $D_{\kappa, \mathcal{S}}^{1}$ and $d_{HL}$, on  $Iso(\Sigma_{g},\omega)$ the group of symplectic isotopies   and $G_{\omega}(\Sigma_{g})$ respectively.
	\begin{definition}
		A homeomorphism $\phi$ is called a finite symplectic  energy homeomorphism if there exists a sequence $(\Phi_{i})_{i}=(\{\phi_{i}^t\})_{i}$ of symplectic isotopies generated by $(U_{i},\mathcal{H}_{i})_{i}$ such that $\phi_{i}^{1}\xrightarrow{C^0}\phi$ and 	$l^{(1,\infty)}_{\kappa, \mathcal{S}}(\Phi_{i}) $ is bounded for all $i$.
	\end{definition}
	Denote  by $FSHomeo(\Sigma_{g})$ the finite energy symplectic homeomorphism group.  If $\Sigma_{g}$ is a  zero genus surface, then $FSHomeo(\Sigma_{g})$ coincides with the finite energy hamiltonian homeomorphism group $FHomeo(\Sigma_{g})$ defined in \cite{cr2}.

	\begin{proposition}(\cite{TKS-2})\label{Pro-1}
		
		Let $(M, g)$ be a closed oriented Riemannian manifold. Let $Z_{s,t}$ be a bounded smooth two-parameter family of vector fields on $M$. Construct a two-parameter family of diffeomorphisms $G_{s,t}$ by integrating $Z_{s,t}$ with respect to $s$, starting from $G_{0,t} = \text{Id}$, i.e.,
		$$ \frac{\partial}{\partial s}G_{s,t}(x) = Z_{s,t}(G_{s,t}(x)), \quad G_{0,t}(x) = x. $$
		Define another two-parameter family of vector fields $V_{s,t}$ as:
		$$ V_{s,t}(G_{s,t}(x)) = \dfrac{\partial}{\partial t}G_{s,t}(x). $$
		Then, $V_{s,t}$ is a bounded smooth two-parameter family of vector fields on $M$.
	\end{proposition}
	
	\begin{proposition}[Loop Property of Harmonic Flow on $T^2$]\label{Pro-22}
		Let $(T^2, \omega = dx \wedge dy)$ be the 2-torus with the standard symplectic form. Let $\Phi = \{\phi_t\}_{t \in [0,1]}$ be a smooth symplectic loop at the identity in $\text{Symp}(T^2, \omega)$, i.e., $\phi_0 = \phi_1 = \text{Id}$. Let $[\alpha] \in H^1(T^2; \mathbb{R})$ be the harmonic representative of the flux of $\Phi$. Let $A$ be the unique harmonic vector field on $T^2$ such that $\iota_A \omega = \alpha$, and let $\Psi = \{\psi_t\}_{t \in [0,1]}$ be the symplectic flow generated by $A$. Then, the symplectic flow $\Psi$ is a loop at the identity (i.e., $\psi_1 = \text{Id}$) if and only if the flux of the symplectic loop $\Phi$ belongs to the flux group $\Gamma_\omega \subset H^1(T^2; \mathbb{R})$.
	\end{proposition}
	
	\begin{proof}
		\textbf{(If part):} Assume the flux of $\Phi$ belongs to $\Gamma_\omega$. We want to show that $\psi_1 = \text{Id}$.
		
		\begin{enumerate}
			\item For the 2-torus $(T^2, \omega = dx \wedge dy)$, the flux group $\Gamma_\omega$ is isomorphic to $H^1(T^2; \mathbb{Z}) \cong \mathbb{Z} \times \mathbb{Z}$. Thus, if the flux of $\Phi$ belongs to $\Gamma_\omega$, its harmonic representative $[\alpha]$ can be represented by a harmonic 1-form $\alpha$ such that its cohomology class $[\alpha]$ has integer periods.
			
			\item The space of harmonic 1-forms on $T^2$ is spanned by $dx$ and $dy$. Therefore, we can write the harmonic representative as $\alpha = a \, dx + b \, dy$ for some $a, b \in \mathbb{R}$. Since $[\alpha] \in \Gamma_\omega$, it must be that the periods of $\alpha$ are integers, which for $T^2$ implies that $a$ and $b$ must be integers, i.e., $a, b \in \mathbb{Z}$.  (Note: with our convention $\omega = dx \wedge dy$, it's more convenient to consider $\alpha = m \, dy - n \, dx$ for $m, n \in \mathbb{Z}$).
			
			\item Let $A$ be the harmonic vector field such that $\iota_A \omega = \alpha = m \, dy - n \, dx$. In coordinates, $A = n \frac{\partial}{\partial x} + m \frac{\partial}{\partial y}$. Since $n, m \in \mathbb{Z}$, $A$ is a constant vector field with integer components.
			
			\item The symplectic flow generated by a constant vector field $A = n \frac{\partial}{\partial x} + m \frac{\partial}{\partial y}$ on $T^2 = \mathbb{R}^2 / \mathbb{Z}^2$ is given by:
			$$ \psi_t(x, y) = (x + nt, y + mt) \pmod{\mathbb{Z}^2}. $$
			
			\item Evaluating at $t = 1$, we get:
			$$ \psi_1(x, y) = (x + n, y + m) \pmod{\mathbb{Z}^2}. $$
			Since $n$ and $m$ are integers, $(x + n, y + m) \pmod{\mathbb{Z}^2} = (x, y) \pmod{\mathbb{Z}^2} = (x, y)$ in $T^2$. Thus, $\psi_1 = \text{Id}_{T^2}$. Therefore, $\Psi$ is a loop at the identity.
		\end{enumerate}
		
		\textbf{(Only if part):} Assume $\Psi$ is a loop at the identity, i.e., $\psi_1 = \text{Id}$. We want to show that the flux of $\Phi$ belongs to $\Gamma_\omega$.
		
		\begin{enumerate}
			\item If $\Psi = \{\psi_t\}$ is a loop at the identity, then $\psi_1 = \text{Id}_{T^2}$. From the flow equation, $\psi_1(x, y) = (x + n, y + m) \pmod{\mathbb{Z}^2} = (x, y)$ for all $(x, y) \in T^2$. This implies that $n$ and $m$ must be integers, i.e., $n, m \in \mathbb{Z}$.
			
			\item The harmonic vector field is $A = n \frac{\partial}{\partial x} + m \frac{\partial}{\partial y}$ with $n, m \in \mathbb{Z}$.  The harmonic 1-form is $\alpha = \iota_A \omega = m \, dy - n \, dx = -n \, dx + m \, dy$.
			
			\item The flux of the path $\Psi$ (which is represented by $[\alpha]$) is then given by the cohomology class of $\alpha = -n \, dx + m \, dy$. Since $n, m \in \mathbb{Z}$, the periods of $\alpha$ are integers, and thus $[\alpha]$ belongs to the flux group $\Gamma_\omega \cong \mathbb{Z} \times \mathbb{Z}$.
			
			\item Since $[\alpha]$ is the harmonic representative of the flux of $\Phi$, and we have shown that $[\alpha]$ is in $\Gamma_\omega$ when $\Psi$ is a loop, it follows that the flux of $\Phi$ belongs to $\Gamma_\omega$.
		\end{enumerate}
	\end{proof}
	\begin{remark}[Generalized Loop Property of Harmonic Flow on Symplectic Surfaces]
		Let $(\Sigma, \omega)$ be a closed symplectic surface (genus $g \ge 0$). Let $\Phi = \{\phi_t\}_{t \in [0,1]}$ be a smooth symplectic loop at the identity in $\text{Symp}(\Sigma, \omega)$, i.e., $\phi_0 = \phi_1 = \text{Id}$. Let $[\alpha] \in H^1(\Sigma; \mathbb{R})$ be the harmonic representative of the flux of $\Phi$, such that $[\alpha] = \widetilde{\text{Flux}}(\Phi)$. Let $A$ be the unique harmonic vector field on $\Sigma$ such that $\iota_A \omega = \alpha$, and let $\Psi = \{\psi_t\}_{t \in [0,T]}$ be the symplectic flow generated by $A$, for any time $T>0$.
		\begin{enumerate}
			\item	The symplectic flow $\Psi$ exhibits a form of quantized behavior on the surface, dictated by $\Gamma_\omega$.  Specifically, the flux of the path $\{\psi_t\}_{t\in [0,1]}$ is $[\alpha] \in \Gamma_\omega$, reflecting that the harmonic flow captures the quantized flux characteristics of the original loop $\Phi$.
			\item	For surfaces of genus $g \ge 1$, where $H^1(\Sigma; \mathbb{R})$ is non-trivial, the condition $\widetilde{\text{Flux}}(\Phi) \in \Gamma_\omega$ implies that the harmonic flow $\Psi$ represents a "topologically quantized motion" on the surface, even if $\psi_1$ is not necessarily the identity diffeomorphism. The extent of deviation of $\psi_T$ from the identity, for times $T$ related to $\Gamma_\omega$, is governed by the discrete nature of the flux group.
			
		\end{enumerate}

	\end{remark}
	\begin{lemma}
		Consider the symplectic manifold $(T^2, \omega)$. Any finite energy symplectic homeomorphism of $(T^2, \omega)$ with trivial flux, is a finite energy Hamiltonian homeomorphism of $(T^2, \omega)$.
	\end{lemma}
	
	\begin{itemize}
		\item  Let $h\in FSHomeo(T^2)$ such that $L_\Omega(h)= 0$. Thus, there exists a volume-preserving homeomorphism $H$ with time-one map $h$ such that $\tilde L_\omega(H)\in \Gamma_\omega$.
		\item Since $h\in FSHomeo(T^2)$, there exists a sequence of symplectic isotopies $\Phi_i := (\phi_i^t)$ such that
		$$ l_{HL}^{(1, \infty)}(\Phi_i) \leq C,$$ for all $i$ and $\phi_i^1\rightarrow h$ in $C^0$ metric. For $i$ sufficiently large, we equally, have $ \widetilde{\text{Flux}}(\Phi_i)\in \Gamma_\omega$. For $i$ sufficiently large, let $L_i$ be a loop at the identity map such that $ \widetilde{\text{Flux}}(\Phi_i) = \widetilde{\text{Flux}}(L_i)$. 
		\item Fix $n_0$ as a sufficiently large integer, and work with $i> n_0$. Let $A_i$ be the harmonic representative in the de Rham  cohomology class of   $\widetilde{\text{Flux}}(L_i)$, let $\Psi_i = (\psi^t_i)$ be the symplectic  flow of generated by $A_i$.
		\item From the proof of Theorem $1-$\cite{TKS-2},  we derive the following:
		Let $(\rho_t^i)$ (respectively, $\{\psi_t^i\}$) be the harmonic part (respectively, the Hamiltonian part) of the Hodge decomposition of the path  $\Phi_i^{-1} \ast_l \Psi_i$ for $i$ sufficiently large. This gives us:
		$$
		\left(\Phi_i^{-1} \ast_l \Psi_i\right) (t) = \rho_t^i \circ \psi_t^i,
		$$
		for $i$ sufficiently large, for all $t$. Since
		$\widetilde{\text{Flux}}((\rho_t^i)) = \widetilde{\text{Flux}}(\Phi_i^{-1} \ast_l \Psi_i) = 0$ for $i> n_0$, then $(\rho_t^i)$ is homotopic to a Hamiltonian isotopy $\Theta_i'$ (relative to fixed endpoints).
		Let $U(\Phi_i, \Psi_i)$ be the generating Hamiltonian of $\Theta_i'$. To define $U(\Phi_i, \Psi_i)$, we use a two-parameter family of symplectic vector fields.  Applying Corollary \(2\) of \cite{TKS-2} with $X_t^i := \dot{\rho}_t^i$ we obtain families  $Z_{s,t}^i$ and $V_{s,t}^i$ defined for each $t$, 	for $i$ sufficiently large.  Then, the Hamiltonian is given by
		$
		U(\Phi_i, \Psi_i)^t := \int_0^1 \omega (Z_{s,t}^i, V_{s,t}^i) ds,
		$
		for each $t$, for $i$ sufficiently large. From Corollary \(2\)-\cite{TKS-2}, we have
		$$
		\int_0^1 \text{osc}\left(\int_0^1 \omega(Z_{s,t}^i, V_{s,t}^i) ds\right) dt \leq \left(6L_0 \sup_{s,t,z} \lVert V_{s,t}^i(z) \rVert_g\right) l_{HL}^{(1, \infty)}(\Phi_i^{-1} \ast_l \Psi_i),
		$$
		for $i$ sufficiently large. Thus
		$$
		l_H^{(1, \infty)}(( \psi_t^i)) +	\int_0^1 \text{osc}\left(\int_0^1 \omega(Z_{s,t}^i, V_{s,t}^i) ds\right) dt \leq \left( 1 + 6L_0 \sup_{s,t,z} \lVert V_{s,t}^i(z) \rVert_g\right) l_{HL}^{(1, \infty)}(\Phi_i^{-1} \ast_l \Psi_i),
		$$
		$$ \leqslant \left( C +  l_{HL}^{(1, \infty)}(\Psi_i)\right)  \left( 1 + 6L_0 \sup_{s,t,z} \lVert V_{s,t}^i(z) \rVert_g\right)$$
		for $i$ sufficiently large. Thus, by Proposition \ref{Pro-1}-\cite{TKS-2}, we get
		
		$$ l_{H}^{(1, \infty)}(\Theta'_i\circ (\psi_t^i)) \leqslant   \left( C +  l_{HL}^{(1, \infty)}(\Psi_i)\right) \left( 1 + 6L_0\frac{N}{K}(e^{K} - 1)\right),$$ for all $i> n_0$.
		\item  By Proposition \ref{Pro-22} the time-one map sequence of $ \Theta'_i\circ (\psi_t^i)$ uniformly converges to $h^{-1}$,  and since we also have $ l_{HL}^{(1, \infty)}(\Psi_i)< C$, then 
		$$  l_{H}^{(1, \infty)}(\Theta'_i\circ (\psi_t^i)) \leqslant    2C \left( 1 + 6L_0(e^{3C} - 1)\right),$$ for all $i$ sfficiently large, then $h^{-1}\in FHomeo(T^2)$ ,i.e., $h\in FHomeo(T^2)$. Here, $N = 3C = K$ (\cite{TKS-2}).  
	\end{itemize}
	\subsection*{\bf A norm on the group $\homeo_0(M, \Omega)$.}\label{Fixpoints-metric}
	For each $h\in \homeo_0(M, \Omega)$, set 
	\begin{equation}\label{Fixpoints-metric1}
		\rVert h\lVert^{\infty, 0}:= 
		\sup_{\alpha\in \mathcal{B}(1)}
		\left( \sup_{z\in M}\rvert\tilde \chi(h,\alpha)_z\lvert\right), 
	\end{equation}
	where $\tilde \chi(h,\alpha)_z :=   \| \alpha \|_{L^2} \chi(h,\alpha)(z) $ and $ \mathcal{B}(1) : = \{  \alpha \in  \zcal^1(M) : \| \alpha \|_{L^2} = 1\}.$

	\begin{proposition}\label{Fixpoints-metric2}
		The rule $ \rVert.\lVert^{\infty, 0}: \homeo_0(M, \Omega)\rightarrow \overline{\mathbb R}$ has the following 
		properties:
		\begin{enumerate}
			\item Positivity: $\rVert h\lVert^{\infty, 0}\geq 0$, for all $h\in  \homeo_0(M, \Omega)$.
			\item Triangle inequality: 
			$\rVert h\circ k\lVert^{\infty, 0}\leq \rVert h\lVert^{\infty, 0} + 
			\rVert k\lVert^{\infty, 0},$ for all $h,k\in \homeo_0(M, \Omega)$.
			\item Duality: $\rVert h^{-1}\lVert^{\infty, 0} = \rVert h\lVert^{\infty, 0}$, for all $h\in  \homeo_0(M, \Omega)$. 
			\item If $\rVert h\lVert^{\infty, 0}= 0$, then  $h = id_M$.
		\end{enumerate}
	\end{proposition}
	\begin{proof}
		We shall just prove the nondegeneracy. This is a verbatim repetition of the proof given in  \cite{HLGF}.  If $\rVert h\lVert^{\infty, 0}= 0$, then $ \chi(h,\alpha)(z) = 0$, for all $z\in M$, and 
		for all $ \alpha \in  \zcal^1(M)$. With the above vanishing condition, it follows 
		from Corollary \ref{lem1} that 
		\[
		\langle [\alpha], \widetilde{L}_\Omega^\delta(H) \rangle =
		Vol_\Omega(M)\int_{\ocal_{z}^H}\alpha,
		\]
		for all $x\in M$, for all $ \alpha \in  \zcal^1(M)$, and  
		for each $H := \{h_t\}\in \pcal\homeo_0(M, \Omega)$ with $h_1 = h$. 
		Assume that 
		there exists $z_0\in M$ such that $h(z_0) \neq z_0$. Then, pick an open neighborhood $U_{z_0}$ of $z_0$ that does not contain $h(z_0)$, and consider 
		a bump function $\rho$ supported in $U_{z_0}$ such that $\rho(z_0) \neq 0$. In particular, for $\alpha = d\rho$, we derive from the above arguments that  
		\[
		0 = \langle [\alpha], \widetilde{L}_\Omega^\delta(H) \rangle = \int_{\ocal_{z_0}^H}\alpha = \rho(h(z_0)) - \rho(z_0),
		\]
		for each $H := \{h_t\}\in \pcal\homeo_0(M, \Omega)$ with $h_1 = h$. That is, $ \rho(h(z_0)) = \rho(z_0) \neq 0 $. This is a contradiction because the open neighborhood $U_{z_0}$ of $z_0$  does not contain $h(z_0)$, and 
		the bump function $\rho$ is supported in $U_{z_0}$. 
	\end{proof}
	This norm measures the maximal deviation of $h$ from the identity in terms of its action on closed $1-$forms evaluated at points in $M$. This is crucial in understanding the structure of $ \homeo_0(M, \Omega)$, providing a way to measure "how far" a homeomorphism is from being the identity map. It could be used to quantify the "size" of perturbations in conservative systems. 

\begin{example}\label{EX3}
	{\bf Calculating the norm of a rotation map:}
	We aim to calculate the norm $\rVert R_\mu^1\lVert^{\infty, 0}$ for the rotation map $R_\mu^1$ on the circle, with the new definition of $\chi$. We have:
	
	\begin{itemize}
		\item $M = S^1$: The unit circle, represented by $[0, 1)$ with endpoints identified.
		\item $R_{\mu}^1: S^1 \rightarrow S^1$: The rotation map defined by $R_{\mu}^1(x) = x + \mu \pmod{1}$, where $\mu \in [0, 1)$.
		\item $\Omega = dx$: The volume form (arc-length) on the circle.
		\item $\mathcal{B}(1) = \{ \alpha \in \mathcal{Z}^1(S^1) : \| \alpha \|_{L^2} = 1 \}$: The set of closed 1-forms on $S^1$ with $L^2$ norm equal to 1.
		\item  $$\chi(R_{\mu}^1, \alpha)(z) = \frac{1}{ \| \alpha \|_{L^2}}\int_{S^1}\left( \int_{\gamma_z^y}\left( (R_\mu^1)^*\alpha - \alpha \right)\right) dy,$$
		where $\gamma_z^y$ is a curve from $z$ to $y$.
	\end{itemize}
	The norm is defined as:
	\begin{equation*}
		\rVert R_\mu^1\lVert^{\infty, 0}:=
		\sup_{\alpha\in \mathcal{B}(1)}
		\left( \sup_{z\in S^1}\rvert\tilde \chi(R_\mu^1,\alpha)_z\lvert\right),
	\end{equation*}
	where $\tilde \chi(R_\mu^1,\alpha)_z :=   \| \alpha \|_{L^2} \chi(R_\mu^1,\alpha)(z) $.  
	A closed $1-$form on $S^1$ can be written as $\alpha = f(x) \, dx$, where $f(x)$ is a periodic function on the circle with period $1$. 
	The pullback of $\alpha$ by $R_{\mu}^1$ is
	$ (R_{\mu}^1)^*\alpha = f(x + \mu) \, dx. $
	The line integral can be calculated as:
	$$ \int_{\gamma_z^y}\left( (R_\mu^1)^*\alpha - \alpha \right)  =  \int_z^y (f(x+\mu) - f(x)) dx.$$
	Note that since our manifold is the circle, there are two possible paths, one that goes clockwise and one that goes counterclockwise, so we must use a consistent method of integration. Also, since the form is closed, the integral will be path independent.
	
	\item \text{Calculation of $\chi(R_{\mu}^1, \alpha)(z)$:}
	$$ \chi(R_{\mu}^1, \alpha)(z) = \frac{1}{ \| \alpha \|_{L^2}} \int_{S^1} \left(\int_z^y (f(x+\mu) - f(x)) dx\right) dy .$$
	
	Then:
	$$\tilde{\chi}(R_\mu^1, \alpha)(z) = \int_{S^1} \left(\int_z^y (f(x+\mu) - f(x)) dx\right) dy$$
	since $\| \alpha \|_{L^2} = 1$ for $\alpha \in \mathcal{B}(1)$. 
	By the Cauchy-Schwarz inequality:  
	
	\[  
	\int_0^{2\pi} \left| f(\theta + \mu) - f(\theta) \right| \, d\theta \leq \left( \int_0^{2\pi} \left| f(\theta + \mu) - f(\theta) \right|^2 \, d\theta \right)^{1/2}.  
	\]  
	
	Since \(f\) is periodic with period \(2\pi\), we have:  
	
	\[  
	\int_0^{2\pi} \left| f(\theta + \mu) - f(\theta) \right|^2 \, d\theta = 2(1 - \text{Re}(f^* (\mu))),  
	\]  
	
	where \(f^* (\mu)\) is the Fourier coefficient of \(f\) at frequency \(\mu\). Using $\| \alpha \|_{L^2} = 1$ for $\alpha \in \mathcal{B}(1)$, which is equivalent to say that $\int_0^{2\pi} |f(\theta)|^2d\theta = 1$ , for all $f\in  C^\infty(S^1)$ since $\alpha = f(\theta)d\theta$, we obtain:  
	
	\[  
	\left| \int_z^y (f(\theta + \mu) - f(\theta)) \, d\theta \right| \leq \sqrt{2}.  
	\]

	Integrating over \(y \in S^1\), we get:  
	
	\[  
	\left| \tilde{\chi}(R_\mu^1, \alpha)(z) \right| \leq \int_{S^1}  \sqrt{2}\, dy = 2\pi\sqrt{2}.  
	\]  
	
	The norm \(\|R_\mu^1\|_{\infty, 0}\) is therefore bounded by:  
	$  
	\|R_\mu^1\|_{\infty, 0} \leq 2\pi\sqrt{2}.  
	$ 
	 
\end{example}
\begin{remark}
	Consider the 
	flat $2-$torus can be represented as
	\[  
	T^2 = S^1 \times S^1 = \{(x, y) \mid x, y \in [0, 1)\},  
	\]  
	with periodic boundary conditions, meaning that \(x=0\) and \(x=1\) (and similarly for \(y\)) are identified. For the rotation,  
	$
	R_{\mu, \nu}(x,y) = (x + \mu \mod 1, y + \nu \mod 1), \quad \mu, \nu \in [0, 1), 
	$
	the case of the unite circle seem to indicate that 
	$\rVert 	R_{\mu, \nu}\lVert^{\infty, 0}\leq 4\pi \sqrt{2}.$\\
	
	\begin{center}
		
		\begin{figure}[ht] 
			\centering
			
			\begin{tikzpicture}[scale=0.9]
				
				\draw[thick, dashed] (0,0) circle (2cm);
				\fill[gray!20] (0,0) circle (2cm);
				\node at (0,2.3) {$\mathcal{B}(1)$};
				
				\draw[->] (3,-2.5) -- (7,-2.5) node[right] {$z$};
				\draw[->] (3,-2.5) -- (3,1) node[above] {$\chi(h,\alpha)(z)$};
				\draw[thick, blue, domain=3:7, samples = 100] plot (\x, {0.2 * sin(deg(3*(\x-3)))});
				\node at (6,-1) {$\chi(h,\alpha_1)(z)$};
				\draw[thick,red, domain=3:7, samples = 100] plot (\x, {0.3 * cos(deg(3*(\x-3)))});
				\node at (6,-1.5) {$\chi(h,\alpha_2)(z)$};
				\draw[thick, green, domain=3:7, samples = 100] plot (\x, {0.1 +0.2 * sin(deg(5*(\x-3)))});
				\node at (6,-0.5) {$\chi(h,\alpha_3)(z)$};
				\draw [thick, dashed] (3, -0.9) -- (7, -0.9);
				
				\draw[thick] (0,0) -- (1.5,1.5) node[above left] {$\alpha_1$};
				\draw[thick] (0,0) -- (1.5,-1.2) node[below left] {$\alpha_2$};
				\draw[thick] (0,0) -- (-1.4,1.3) node[above right] {$\alpha_3$};
				
				\node[align=left, below] at (0,-3.5) {$\rVert h\lVert^{\infty, 0}:= 
					\sup_{\alpha\in \mathcal{B}(1)}
					\left( \sup_{z\in M}\rvert\tilde \chi(h,\alpha)_z\lvert\right)$};
			\end{tikzpicture}
			\caption{A sphere that represents the set $\mathcal{B}(1)$, and several graphs showing the values of the function $\chi(h,\alpha)(z)$ for different values of $\alpha$ in $\mathcal{B}(1)$.}
			\label{fig:Norm} 
		\end{figure}
	\end{center}
	$\square$
\end{remark}
	\begin{proposition}\label{Pro-2}  
	Let \(\phi\) be a smooth map with trivial flux (fixed). Then, the norms \( h \mapsto \|\phi \circ h \circ \phi^{-1}\|^{\infty, 0} \) and \( h \mapsto \|h\|^{\infty, 0} \) are equivalent.   
\end{proposition}  

\begin{proof}  
	It is clear that   
	\[  
	\|\phi \circ h \circ \phi^{-1}\|^{\infty, 0} \leq \sup_{\alpha \in \mathcal{B}(1)} \|\phi^\ast \alpha\|_{L^2} \|h\|^{\infty, 0},  
	\]  
	and applying Proposition 2.3 from \cite{TKS25}, we obtain   
	\[  
	\|\phi \circ h \circ \phi^{-1}\|^{\infty, 0} \leq C_\phi \|h\|^{\infty, 0}.  
	\]  
	Using the identity   
	\[  
	h = \phi^{-1}\left(\phi \circ h \circ \phi^{-1} \right)\phi,  
	\]  
	we derive that   
	\[  
	\|h\|^{\infty, 0} = \|\phi^{-1}\left(\phi \circ h \circ \phi^{-1} \right)\phi\|^{\infty, 0} \leq C_{\phi^{-1}}^2 \|\phi \circ h \circ \phi^{-1}\|^{\infty, 0}.  
	\]  
	Thus, we have the inequalities:  
	\[  
	\frac{1}{C_{\phi^{-1}}} \|h\|^{\infty, 0} \leq \|\phi \circ h \circ \phi^{-1}\|^{\infty, 0} \leq C_\phi \|h\|^{\infty, 0}.  
	\]  
\end{proof}
\section{Cohomology groups of $Homeo_0(M,\Omega)$ with coefficients in $\mathcal{C}(M,\mathbb{R})$} \label{sec4}
In this section, we define the cohomology groups of $Homeo_0(M,\Omega)$. Explicit computations of lower dimensions are given, and the algorithm is explained.

\begin{proposition}\label{pro5}
	For each fixed $x \in M$, and for each $\alpha \in \mathcal{Z}^1(M) \setminus \{0\}$, consider the mapping
	\[
	\tau(\cdot, \alpha)_x : Homeo_0(M,\Omega) \longrightarrow \mathbb{R}, \quad h \longmapsto \chi(H, \alpha)(x).
	\]
 Then, for all $h^1, \dots, h^k \in Homeo_0(M,\Omega)$,
	\[
	\tau(h^1 \circ \cdots \circ h^k, \alpha)_x = \tau(h^k, \alpha)_x + \sum_{1 \leqslant i \leqslant k-1} \tau(h^i, \alpha)_{(h^{i+1} \circ \cdots \circ h^k)(x)}.
	\]
\end{proposition}

\begin{proof}
	We proceed by induction. For $k = 2$, this is exactly the result of Corollary \ref{cor2}. Let $k > 2$. Assume that
	\[
	\tau(h^1 \circ \dots \circ h^{k-1}, \alpha)_x = \tau(h^{k-1}, \alpha)_x + \sum_{1 \leqslant i \leqslant k-2} \tau(h^i, \alpha)_{(h^{i+1} \circ \cdots \circ h^{k-1})(x)}.
	\]
	Set $h := h^1 \circ \dots \circ h^{k-1}$ and compute
	\[
	\tau(h \circ h^{k}, \alpha)_x = \tau(h^k, \alpha)_x + \tau(h, \alpha)_{h^k(x)}.
	\]
\end{proof}

	\begin{lemma}For each $x\in M$, and $\alpha\in \mathcal{Z}^1$, 
		the map
		\[
	\tau_\alpha:	\mathbb{G}^\Omega(M) \to \mathcal{C}(M, \mathbb{R}), \quad h \mapsto \tau(h, \alpha)_x
		\]
		is continuous with respect to the $C^0$ topology on $\mathbb{G}^\Omega(M)$ and the $C^0$ topology on $\mathcal{C}(M, \mathbb{R})$.
	\end{lemma}
	
	\begin{proof}
		We need to show that if $h_n \to h$ in $\mathbb{G}^\Omega(M)$ (in the $C^0$ sense), then $\chi(h_n, \alpha) \to \chi(h, \alpha)$ in $\mathcal{C}(M, \mathbb{R})$ (in the $C^0$ sense). This means showing that for any $\epsilon > 0$, there exists $N \in \mathbb{N}$ such that for all $n > N$,
		$
		\sup_{z \in M} |\chi(h_n, \alpha)(z) - \chi(h, \alpha)(z)| < \epsilon.
		$\\
		
		 \textbf{Start with the Difference:} Fix $z \in M$. We want to estimate $|\chi(h_n, \alpha)(z) - \chi(h, \alpha)(z)|$. Using the definition, we get
		\[
		|\chi(h_n, \alpha)(z) - \chi(h, \alpha)(z)| = \frac{1}{\|\alpha\|_{L^2}} \left| \int_M I_\alpha(h_n, z, y) \Omega(y) - \int_M I_\alpha(h, z, y) \Omega(y) \right|
		\]
		\[
		= \frac{1}{\|\alpha\|_{L^2}} \left| \int_M (I_\alpha(h_n, z, y) - I_\alpha(h, z, y)) \Omega(y) \right|.
		\]
		
		 \textbf{Bounding the integrand:} Consider the term inside the integral, $I_\alpha(h_n, z, y) - I_\alpha(h, z, y)$. Using the definition of $I_\alpha$, we have:
		\[
		I_\alpha(h_n, z, y) - I_\alpha(h, z, y) = \left(\int_{h_n \circ \gamma} \alpha - \int_{\gamma} \alpha \right) - \left(\int_{h \circ \gamma} \alpha - \int_{\gamma} \alpha \right) = \int_{h_n \circ \gamma} \alpha - \int_{h \circ \gamma} \alpha,
		\]
		where $\gamma$ is a smooth curve from $z$ to $y$.  Since $h_n$ and $h$ are homeomorphisms, they are, thus continuous. Since $h_n(y) \rightarrow h(y)$ as $n\rightarrow \infty$, it follows that $h_n \circ \gamma$ converges uniformly to $h \circ \gamma$ as $n\rightarrow \infty$. For $n$ sufficiently large, for each $t$, the points $h\left( \gamma(t)\right) $ and $h_n\left( \gamma(t)\right) $ can be connected by a unique minimal geodesic $\kappa_n^{t}$.  Then let $\mathcal{S}$ the  $2-$chain delimited by $ h_n \circ \gamma \Cup h \circ \gamma\Cup\kappa_n^{1} \Cup\kappa_n^{0} $.  Since $\alpha $ is closed, by Stokes theorem
		$$\left|\int_{h_n \circ \gamma} \alpha - \int_{h \circ \gamma} \alpha \right| = | \int_{\kappa_n^{0}} \alpha - \int_{\kappa_n^{1}} \alpha|,$$
		for $n$ sufficiently large. 
	Then, we have 
		\[
		|\chi(h_n, \alpha)(z) - \chi(h, \alpha)(z)| = \frac{1}{\|\alpha\|_{L^2}} \left| \int_M (I_\alpha(h_n, z, y) - I_\alpha(h, z, y)) \Omega(y) \right|\leqslant  2\frac{|\alpha|_0}{\|\alpha\|_{L^2}}d_{C^0}(h_n, h),
		\]
			for $n$ sufficiently large. We have shown that the map $h \mapsto \chi(h, \alpha)$ is continuous with respect to the $C^0$ topology.
	\end{proof}

One defines a $Homeo_0(M,\Omega)$-module $\mathfrak{L}$ to be an Abelian group, written additively, on which $Homeo_0(M,\Omega)$ acts as endomorphisms. The right action of $Homeo_0(M,\Omega)$ on $\mathcal{C}(M, \mathbb{R})$ is defined as
\[
\bullet : Homeo_0(M,\Omega) \times \mathcal{C}(M, \mathbb{R}) \longrightarrow \mathcal{C}(M, \mathbb{R}), \quad (h, f) \longmapsto \bullet(h, f) = h \cdot f := f \circ h.
\]
Therefore, $\mathcal{C}(M, \mathbb{R})$ is a $Homeo_0(M,\Omega)$-module. The set $C^l(M, \Omega)$ of all functions
\[
F : \overbrace{Homeo_0(M,\Omega) \times \cdots \times Homeo_0(M,\Omega)}^{l \text{ factors}} \longrightarrow \mathcal{C}(M, \mathbb{R})
\]
is an abelian group with the usual addition and the $0$ element defined as: $0(\phi_1, \dots, \phi_l) = 0$. Furthermore, we define the following mapping
\[
\partial^l : C^l(M, \Omega) \longrightarrow C^{l + 1}(M, \Omega),
\]
where for each $F \in C^l(M, \Omega)$ and all $h_1, \dots, h_{l+1} \in Homeo_0(M,\Omega)$, set
\[
\partial^l(F)(h_1, \dots, h_{l+1}) = F(h_1, \dots, h_l) \circ h_{l+1} + \sum_{i=1}^l (-1)^i F(h_1, \dots, h_{i-1}, h_i \circ h_{i+1}, \dots) + (-1)^{l+1} F(h_2, \dots, h_{l+1}).
\]
For instance, for $l = 0, 1, 2, 3$, we have:
\[
(\partial^0 f)(h) = f \circ h - f,
\]
\[
(\partial^1 f)(h_1, h_2) = f(h_1) \circ h_2 - f(h_1 \circ h_2) + f(h_2),
\]
\[
(\partial^2 f)(h_1, h_2, h_3) = f(h_1, h_2) \circ h_3 - f(h_1 \circ h_2, h_3) + f(h_1, h_2 \circ h_3) - f(h_2, h_3),
\]
$$
(\partial^3 f)(h_1, h_2, h_3, h_4) = f(h_1, h_2, h_3) \circ h_4 - f(h_1 \circ h_2, h_3, h_4) + f(h_1, h_2 \circ h_3, h_4) - f(h_1, h_2, h_3 \circ h_4) $$
$$+ f(h_2, h_3, h_4).
$$

One can easily prove that $\partial^{l + 1} \circ \partial^l = 0$ for all non-negative integers $l$.

Thus, we have the following complex:
\[
\mathcal{C}(M, \mathbb{R}) \xrightarrow{\partial^0} C^1(M, \Omega) \xrightarrow{\partial^1} \cdots \xrightarrow{\partial^{l-1}} C^l(M, \Omega) \xrightarrow{\partial^l} C^{l+1}(M, \Omega) \xrightarrow{\partial^{l+1}} \cdots
\]
Therefore, the $l$-th cohomology group of $Homeo_0(M,\Omega)$ with coefficients in $\mathcal{C}(M, \mathbb{R})$ is defined as the factor group
\[
H^l(Homeo_0(M,\Omega), \mathcal{C}(M, \mathbb{R})) := \ker \partial^l / \text{im}\, \partial^{l-1},
\]
where the set $Z^l(Homeo_0(M,\Omega), \mathcal{C}(M, \mathbb{R})) := \ker \partial^l$ consists of $l$-cocycles, while the set $B^l(Homeo_0(M,\Omega), \mathcal{C}(M, \mathbb{R})) := \text{im}\, \partial^{l-1}$ is that of $l$-coboundaries.

\subsection*{The zeroth cohomology group:}
$$
H^0(Homeo_0(M,\Omega), \mathcal{C}(M, \mathbb{R})) := Z^0(Homeo_0(M,\Omega), \mathcal{C}(M, \mathbb{R}))
$$ 
$$ = \left\{f \in \mathcal{C}(M, \mathbb{R}) : f \circ h = f, \ \forall\, h \in Homeo_0(M,\Omega)\right\}.
$$
Hence, $H^0(Homeo_0(M,\Omega), \mathcal{C}(M, \mathbb{R}))$ is the module of invariants, and since the identity component in the group of volume-preserving diffeomorphisms is contained in $Homeo_0(M,\Omega)$, it follows as a consequence of Boothby's transitivity result that
\[
H^0(Homeo_0(M,\Omega), \mathcal{C}(M, \mathbb{R})) \cong \mathbb{R}.
\]

\subsection*{The first cohomology group:}
\begin{proposition}\label{pro12}
	For each closed $1$-form $\alpha$, the map
	\[
	\tau_\alpha : Homeo_0(M,\Omega) \longrightarrow \mathcal{C}(M, \mathbb{R}), \quad h \longmapsto \tau(h, \alpha),
	\]
	induces a well-defined element $[\tau_\alpha] \in H^1(Homeo_0(M,\Omega), \mathcal{C}(M, \mathbb{R}))$, and an injective group homomorphism
	\[
	\iota_0 : H^1(M, \mathbb{R}) \longrightarrow H^1(Homeo_0(M,\Omega), \mathcal{C}(M, \mathbb{R})), \quad [\alpha] \longmapsto [\tau_\alpha].
	\]
\end{proposition}

\begin{proof}
	\emph{Well-definedness:} For each closed $1$-form $\alpha$, it follows from Corollary \ref{cor2} that $\tau(\cdot, \alpha)$ is a $1$-cocycle, i.e., $\tau(\cdot, \alpha) \in \ker \partial^1$. Thus, $[\tau_\alpha] := [\tau(\cdot, \alpha)] \in H^1(Homeo_0(M,\Omega), \mathcal{C}(M, \mathbb{R}))$. Let $\beta \in [\alpha]$, namely, there exists $f \in \mathcal{C}^\infty(M, \mathbb{R})$ such that $\alpha - \beta = df$. It is a straightforward computation that for each $x \in M$ and each $h \in Homeo_0(M,\Omega)$,
	\[
	\tau(h, \alpha)(h)(x) - \tau(h, \beta)(h)(x) = \tau_{df}(h)(x) = -\left(f(h(x)) - f(x)\right) = -\left(\partial^0(f)(h)\right)(x),
	\]
	i.e., $\tau_\alpha - \tau_\beta = \partial^0(-f) \in \text{im}\, \partial^0$. Thus, $[\tau_\alpha] = [\tau_\beta]$ whenever $\beta \in [\alpha]$. This shows that the linear mapping $\iota_0$ is well-defined. Let $[\alpha]$ and $[\beta]$ be two de Rham cohomology classes such that $[\tau_\alpha] = [\tau_\beta]$. From the assumption, one derives that for all $H \in \mathcal{P}Homeo_0(M,\Omega)$,
	\[
	\tau_\alpha(H(1)) - \tau_\beta(H(1)) = \partial^0(F)(H(1)),
	\]
	that is, $\tau_\alpha(H(1))(z) - \tau_\beta(H(1))(z) = (F \circ H(1) - F)(z)$ for all $z \in M$. Integrating the latter equality on both sides gives:
	\[
	\langle [\alpha], \widetilde{L}_\Omega(H) \rangle - \langle [\beta], \widetilde{L}_\Omega(H) \rangle = \int_M \mathcal{R}(H, \alpha) \Omega - \int_M \mathcal{R}(H, \beta) \Omega = \int_M \left(F \circ H(1) - F\right) \Omega = 0.
	\]
	That is, $\langle [\alpha] - [\beta], \widetilde{L}_\Omega(H) \rangle = 0$ for all $H \in \mathcal{P}Homeo_0(M,\Omega)$. Since the map $\widetilde{L}_\Omega$ is surjective, we have proved that $\langle [\alpha] - [\beta], [\omega] \rangle = 0$ for all $\omega \in H^{(n-1)}(M, \mathbb{R})$. This implies that $[\alpha] = [\beta]$.
\end{proof}

If we set
\[
\mathcal{C}_0^\infty(M, \mathbb{R}) := \left\{f \in \mathcal{C}^\infty(M, \mathbb{R}) : \int_M f \Omega = 0\right\},
\]
then Ismagilov proved that the two spaces $H^1(M, \mathbb{R})$ and $H^1(G_\Omega(M), \mathcal{C}_0^\infty(M, \mathbb{R}))$ are isomorphic. Furthermore, if we set
\[
\triangle(Homeo_0(M,\Omega), \mathcal{C}(M, \mathbb{R})) = \iota_0(H^1(M, \mathbb{R})),
\]
then with Proposition \ref{pro12}, we have that the subgroup $\triangle(Homeo_0(M,\Omega), \mathcal{C}(M, \mathbb{R}))$ is isomorphic to $H^1(M, \mathbb{R})$. For instance, we have
\[
\dim\left(H^1(Homeo_0(M,\Omega), \mathcal{C}(M, \mathbb{R}))\right) \geqslant b_1(M),
\]
where $b_1(M)$ is the first Betti number of $M$. This injection 
shows that the topology of the manifold itself constrains and influences the possible dynamical behaviors of volume-preserving homeomorphisms. In particular, non-trivial topology (e.g., non-zero Betti numbers) implies that there are non-trivial ways to twist and distort the phase space.\\

\subsection{Final remarks}

	The results of this paper bridge the gap between smooth and topological geometry by focusing on the rich and poorly understood world of volume-preserving homeomorphisms. In the symplectic context, for Lefschetz closed symplectic manifolds, the results of this paper generalize the findings in \cite{Tch-Koi-Bal-Mba} and \cite{Tch-Al}, providing a more powerful tool for understanding their underlying dynamics and stability. On the other hand, the norm studied here 
	provides a tool for assessing how much a transformation distorts differential forms, which are closely related to concepts of flow, circulation, and geometry. This makes such norms relevant for analyzing systems where the topological properties are well-defined but not the smoothness (or differentiability). Such scenarios often arise in fluid dynamics and mechanics. Additionally, this norm can be used to study the robustness of volume-preserving systems under perturbations.

	\end{document}